\documentclass{cocv}
\usepackage{amsmath,amsfonts,array,enumerate,tikz}
\usepackage[dvips]{epsfig}
\usepackage{amssymb,psfrag,pdfsync}
\usepackage{amsrefs}
\newcommand{\ELeq}[2]{\frac{d}{dt}\frac{\partial{#1}}{\partial\dot{#2}}-
	\frac{\partial{#1}}{\partial #2}}

\newcommand{\n}{{\mathbf n}}

\theoremstyle{definition}
\newtheorem{theorem}{Theorem}[section]
\newtheorem{lemma}[theorem]{Lemma}

\newtheorem{proposition}[theorem]{Proposition}
\newtheorem{definition}[theorem]{Definition}
\newtheorem{remark}[theorem]{Remark}

\newtheorem{example}[theorem]{Example}

\DeclareMathOperator{\rank}{rank}

\usepackage{color}
\numberwithin{equation}{section}
\numberwithin{theorem}{section}
\begin{document}
\title{On the Lagrangian Structure of Reduced Dynamics Under Virtual
	Holonomic Constraints}\thanks{A preliminary version of this
	paper has been presented at
	the 9\textsuperscript{th} IFAC Symposium on Nonlinear Control Systems (NOLCOS)~\cite{MohMagCon13_1}.}\thanks{ A. Mohammadi was supported
	in part by the Ontario Graduate Scholarship (OGS) and by the
	Natural Sciences and Engineering Research Council of Canada
	(NSERC).}
\thanks{M. Maggiore was supported by the
	Natural Sciences and Engineering Research Council of Canada
	(NSERC).}
\thanks{M. Maggiore is the corresponding author.}
\author{Alireza Mohammadi}\address{Department of Electrical and Computer
	Engineering, University of Toronto, 10 King's College Road, Toronto,
	ON, M5S 3G4, Canada; alireza.mohammadi@mail.utoronto.ca \& maggiore@ece.utoronto.ca}
\author{Manfredi Maggiore}\sameaddress{1}
\author{Luca Consolini}\address{Dipartimento di Ingegneria
	dell'Informazione, Via Usberti 181/a, 43124 Parma, Italy; lucac@ce.unipr.it}
%
%
\begin{abstract} This paper investigates a class of Lagrangian control systems with $n$
	degrees-of-freedom (DOF) and $n-1$ actuators, assuming that $n-1$
	virtual holonomic constraints have been enforced via feedback, and a
	basic regularity condition holds.  The reduced dynamics of such
	systems are described by a second-order unforced differential
	equation. We present necessary and sufficient conditions under which
	the reduced dynamics are those of a mechanical system with one DOF
	and, more generally, under which they have a Lagrangian structure. In
	both cases, we show that typical solutions satisfying the virtual
	constraints lie in a restricted class which we completely
	characterize. \end{abstract}
%
%
\maketitle
\section*{Introduction}
A virtual holonomic constraint (VHC) is a relation involving the
configuration variables of a mechanical system that can be made
invariant via feedback control.  VHCs emulate the presence of physical
constraints, and can be used to induce desired behaviours.  An early
manifestation of this idea appeared in the work of Nakanishi {\em et
	al.}~\cite{nakanishi-2000}, where the authors enforced, via feedback
control, a constraint on the angles of an acrobot to induce
pendulum-like dynamics imitating the brachiating motion of an ape.

Over the past decade, the idea of VHC rose to prominence with research
on biped robots by J. Grizzle and collaborators (see,
e.g.,~\cite{PleGriWesAbb03,WesGriKod03,WesGriCheChoMor07,
	CheGriShi08}). In this body of work, VHCs are used to encode
different walking gaits, without requiring the design of
time-dependent reference signals for the robot joints. The authors
show that when a suitable VHC is enforced, the resulting constrained
motion exhibits a stable hybrid limit cycle corresponding to a periodic
walking motion.  A similar idea has been used to make snake robots
follow paths on the plane~\cite{mohammadi2014direction,
	mohammadi2015maneuver}. In this context, the VHC encodes a lateral
undulatory gait whose parameters are dynamically adjusted to control
the velocity vector of the snake in such a way that the centre of mass
converges to a desired path.
In~\cite{ShiPerWit05,ShiRobPerSan06,ShiFreGus10,FreRobShiJoh08}, VHCs
are used to plan repetitive motions in mechanical control systems. In
this context, VHCs are used to aid the selection of closed orbits
corresponding to desired repetitive behaviors, which can then be
stabilized in a variety of ways.

In classical mechanics, a Lagrangian system subject to an ideal
holonomic constraint (one with the property that the constraint forces
do not make work on virtual displacements), gives rise to Lagrangian
reduced dynamics whose Lagrangian function is the restriction of the
unconstrained Lagrangian to the constraint manifold.  It is natural to
ask whether an analogous property holds for Lagrangian {\em control}
systems subject to {\em virtual} holonomic constraints.  This paper
investigates this problem and solves it completely for the specific
setup described below.

{\em Contributions of this paper.}   We consider
Lagrangian control systems with $n$ DOF and $n-1$ controls. We assume
that a regular VHC, $h(q)=0$, of order $n-1$ (the definition will be
given in Section~\ref{sec:Prelim}) has been enforced via feedback
control, and we investigate the resulting reduced dynamics. These are
given by a second-order unforced differential equation of the form
\begin{equation}\label{eq:sys:0}
\ddot s = \Psi_1(s) + \Psi_2(s) \dot s^2,
\end{equation}
where either $(s,\dot s)\in \mathbb{R} \times \mathbb{R}$ or $(s,\dot s) \in \mathbb{S}^1
\times \mathbb{R}$.

This paper presents three main results. In
Theorem~\ref{thm:ILPsolution:part1}, it is shown that when the state
space of the reduced dynamics is $\mathbb{R} \times \mathbb{R}$, the reduced
dynamics {\em always} admit a global mechanical structure, i.e.,
equation~\eqref{eq:sys:0} results from the Euler-Lagrange equation
with a Lagrangian function of the form $L(s,\dot s) = (1/2) M(s) \dot
s^2 - V(s)$, with $M > 0$.  When the state space of the reduced
dynamics is the cylinder $\mathbb{S}^1 \times \mathbb{R}$, a Lagrangian structure
may not exist. In Theorem~\ref{thm:ILPsolution:part2} we give explicit
necessary and sufficient conditions guaranteeing that the reduced
dynamics have a global mechanical structure.  In
Theorem~\ref{thm:ILPsolution:part3} we go one step further, and give
necessary and sufficient conditions under which the reduced dynamics
possess any global Lagrangian structure, possibly not in mechanical
form. A byproduct of Theorems~\ref{thm:ILPsolution:part2}
and~\ref{thm:ILPsolution:part3} is that when the state space
of~\eqref{eq:sys:0} is $\mathbb{S}^1 \times \mathbb{R}$, generically {\em there does
	not exist} a global Lagrangian structure.  In addition to these
results, in Section~\ref{sec:MotionCharac} we characterize the
qualitative properties of trajectories of the reduced dynamics.


{\em Related work.}  The results presented in this paper complement
work in~\cite{Maggiore-2012,Consolini-2012}, in which examples were
given showing that the reduced dynamics may possess stable limit
cycles, therefore ruling out the existence of a Lagrangian structure.
In~\cite{Maggiore-2012} sufficient conditions were provided
guaranteeing the existence of a global mechanical structure, but their
necessity was not investigated and more general Lagrangian structures
were not considered.

The inverse problem of calculus of variations (IPCV) is concerned with
finding conditions under which a system of differential equations can
be derived from a variational principle. Comprehensive historical
surveys regarding this problem can be found in \cite{Santilli-1978,
	TontiSurvey, HandbookGlobA}. We will now give an account of some of the
key findings in this field.  In Section~\ref{sec:InvLagProb} (see
Remark~\ref{rem:lit}) we will comment on the fact that the results of
this paper are not contained in the existing literature.

A special case of IPCV, namely, the inverse problem of Lagrangian
mechanics (IPLM), can be traced back to the seminal work of Sonin in
1886~\cite{sonin1886} and Helmholtz in 1887~\cite{Helmholtz-1887}. The
problem investigated in this paper fits within the IPLM framework.
Helmholtz found necessary conditions (today referred to as the
``Helmholtz conditions,''~\cite{Santilli-1978}) under which a given
system of second-order ordinary differential equations is equivalent
to a set of Euler-Lagrange equations derived from some Lagrangian
function. In 1896, Mayer~\cite{Mayer-1896} showed that the Helmholtz
conditions are sufficient as well for the local existence of a
Lagrangian. The Helmholtz conditions are a mixed set of partial
differential equations and algebraic equations in terms of a set of
unknown functions.  It is noteworthy that if these equations can be
solved for a given system of second-order ODE's, the corresponding
Lagrangian is given by the Tonti-Vainberg integral
formula~\cite{Tonti-1969, Vainberg-1959}. Unfortunately, solving the
equations is a nontrivial task. Indeed, the Helmholtz conditions, as
shown by Henneaux~\cite{Henneaux-1982}, are in general strong and
over-determined in the sense that if these conditions admit a
solution, it will be generally unique. For the case of one DOF systems
(i.e., given by one second-order ODE), Darboux~\cite{Darboux-1894}
solved the IPLM in 1894, showing that such systems are always locally
Lagrangian. In 1941, Douglas~\cite{Douglas-1941} could solve the IPLM
for the case of two DOF.  There was a revival of interest in the IPLM
around the 1980's thanks in part to the monograph by
Santilli~\cite{Santilli-1978}. Using the tools of differential
geometry and global analysis, researchers started to encode the
Helmholtz conditions in geometric framework~\cite{Santilli-1978,
	Tonti-1969, Sarlet-1982,Anderson-1980,Crampin-1981, Crampin-1984,
	Takens-1979,Sarlet-1993,Kru97,Kru15}.  The paper by Saunders
\cite{Saunders-2010} reviews the contributions to IPCV since 1979 to
date.

{\em Relevance of ILP in control of mechanical systems.} The reduced
dynamics studied in this paper describe the behavior of any mechanical
system with $n$ degrees-of-freedom and $n-1$ actuators that is under the
influence of $n-1$ virtual holonomic constraints.  Examples include
the acrobot~\cite{nakanishi-2000}, the
pendubot~\cite{consolini2011swing}, Getz's bicycle
model~\cite{Consolini-2012}, and some planar biped robots in their
swing phase, such as RABBIT with $7$
degrees-of-freedom~\cite{chevallereau2003rabbit}. Solving the ILP for
the reduced dynamics is a crucial building block for later development
of control laws for this class of mechanical systems. Indeed, if the
constrained system is Lagrangian, then as we show in this paper the
generic trajectories of the mechanical system under the influence of
VHCs is a trichotomy of oscillations, rotations, and helices (defined
in Section~\ref{sec:qualitative_properties}). Based on the desired
repetitive behavior that the mechanical system should perform, the
designer can then choose from a plethora of closed orbits resulting
from the Lagrangian structure. In the absence of a Lagrangian
structure, closed orbits may not longer be a generic feature of the
constrained dynamics, making it hard or even impossible to impress a
repetitive behaviour on the mechanical system.

{\em Notation.} We let $\n:=\{1,\ldots, n\}$, and given $x \in \mathbb{R}^n$,
we denote $\|x\|:=(x^\top x)^{1/2}$.  Given $x\in \mathbb{R}$ and $T>0$,
then $[x]_T:= x \mbox{ modulo } T$. The set of real numbers modulo $T$
is denoted by ${[\mathbb{R}]_T}$. Therefore, $[\mathbb{R}]_T=\{[x]_T: x\in \mathbb{R}\}$. The set
$[\mathbb{R}]_T$ can be given the structure of a smooth manifold diffeomorphic to
the unit circle $\mathbb{S}^1 \subset \mathbb{C}$ through the map $[x]_T \mapsto
\exp(i  (2\pi/T) [x]_T)$. Given a function $h:\mathcal{Q}\to \mathbb{R}^{k}$, we define
$h^{-1}(0):=\{q\in\mathcal{Q}: h(q)=0\}$.  Given a smooth manifold
$\mathcal{Q}$, we denote by $T \mathcal{Q}$ its tangent bundle, $T\mathcal{Q}:=\{(p,v_p): p
\in \mathcal{Q}, v_p \in T_p \mathcal{Q}\}$. If $h:\mathcal{Q}_1 \to \mathcal{Q}_2$ is a smooth map
between manifolds, and $p\in \mathcal{Q}_1$, $dh_p: T_p \mathcal{Q}_1 \to T_{h(p)}
\mathcal{Q}_2$ denotes the differential of $h$ at $p$, while $dh: T\mathcal{Q}_1 \to T
\mathcal{Q}_2$ denotes the global differential of $h$, defined as $dh:(p,v_p)
\mapsto (h(p),dh_p(v_p))$. If $h:\mathcal{Q}_1 \to \mathcal{Q}_2$ is a diffeomorphism,
then we say that $\mathcal{Q}_1, \mathcal{Q}_2$ are diffeomorphic, and we write $\mathcal{Q}_1
\simeq \mathcal{Q}_2$. In this case, the global differential $dh: T\mathcal{Q}_1 \to T
\mathcal{Q}_2$ is a diffeomorphism as well (see~\cite[Corollary 3.22]{Lee13}).

\section{Introductory example}\label{sec:introductory_example}

Consider a material particle on a plane with inertial coordinates
$q=[q_1 \ \ q_2]^\top \in \mathbb{R}^2$ and unit mass. Assume the particle
is subject to a planar gravitational central force with centre at
$a=[a_1 \ \ a_2]^\top \in \mathbb{R}^2$. Let the gravitational potential be
given by $P(q)=-1 / \|q - a\|$. Suppose a control force $F= B(q) u$ is
exerted on the particle, with $B(q)= q$, where $u \in \mathbb{R}$ is the
control input. The particle model reads
\begin{equation}\label{eq:particle_model}
\ddot q = -\nabla P(q) + B(q) u.
\end{equation}
This is a Lagrangian control system of the form 
\begin{equation}\label{eq:Lagrangian_control_sys}
\ELeq{\mathcal{L}}{q}=B(q)\tau,
\end{equation}
with $\mathcal{L}(q,\dot q) = (1/2) \|\dot q\|^2 - P(q)$.

Pick $b \in \mathbb{R}^2$ such that $\|b\|<1$, and consider the problem of
constraining the motion of the particle on a unit circle centred at
$b$, which corresponds to enforcing the constraint $h(q) = \|q -
b\|-1=0$ via feedback. Setting $e = h(q)$, we have that, along
trajectories of the particle, 
\[
\ddot e = f(q,\dot q) + \frac{ (q-b)^\top B(q)}{\|q-b\|} u,
\]
where $f$ is a smooth function. On the circle $h^{-1}(0)$, the vectors
$q-b$ and $B(q)=q$ are never orthogonal, so the coefficient of $u$ in
$\ddot e$ is nonzero. In other words, the output function $e = h(q)$
has relative degree two on $h^{-1}(0)$. The input-output linearizing
feedback
\[
u(q,\dot q) = \frac{\|q-b\|}{(q-b)^\top B(q)}[-f(q,\dot q) - k_1
e - k_2 \dot e], \ k_1,k_2>0,
\]
asymptotically stabilizes the zero dynamics manifold $\Gamma =
\{(q,\dot q): h(q)=0, d h_q \dot q =0\}$, therefore enforcing the
constraint $h(q)=0$.

We call the relation $h(q)=0$ a {\em virtual holonomic constraint
	(VHC)}, i.e., a holonomic constraint that does not physically exist,
but which can be enforced via feedback control. We call the zero
dynamics manifold $\Gamma$ the {\em constraint manifold} associated
with the VHC $h(q)=0$, and we call the dynamics of the particle on
$\Gamma$ the {\em reduced dynamics.} In this paper we investigate
conditions under which the reduced dynamics possess a Lagrangian
structure, i.e., there exists a function $L: \Gamma \ni (s,\dot s) \to
\mathbb{R}$ such that the reduced dynamics satisfy the Euler-Lagrange
equation 
\[
\frac{d}{dt} \frac{\partial L}{\partial \dot s} - \frac{\partial
	L}{\partial s} =0.
\]
To derive the reduced dynamics of our particle model subject to the
VHC $h(q)=0$, we multiply both sides of~\eqref{eq:particle_model} by a
left-annihilator of $B$,
\[
B^\perp:= B^\top J, \ \ J=\begin{bmatrix} 0 & 1 \\ -1 & 0
\end{bmatrix},
\]
and evaluate the result on $\Gamma$ by picking a parametrization $q=
\sigma(s)$ of the circle $h^{-1}(0)$ and setting
\[
q = \sigma(s):=b+\begin{bmatrix} \cos s \\ \sin s
\end{bmatrix}, \ \ \dot q = \sigma'
\dot s, \ \ \ddot q =\sigma' \ddot s + \sigma'' \dot s^2.
\]
By so doing, we obtain
\begin{equation}\label{eq:reduced_dyn_particle}
\ddot s = -\frac{B^\top J \,\nabla P}{B^\top J
	\sigma'}\Bigg|_{q = \sigma(s)} -\frac{B^\top J \sigma''}{B^\top J
	\sigma'}\Bigg|_{q = \sigma(s)} \dot s^2.
\end{equation}
For each $s$, the vector $J \sigma'(s)$ is orthogonal to the circle
$h^{-1}(0)$ at $\sigma(s)$, so it is proportional to
$(q-b)|_{q=\sigma(s)}$. Since, on $h^{-1}(0)$, the vectors $B(q)$ and
$q-b$ are never orthogonal, we have that $B^\top J \sigma' \neq 0$,
and so~\eqref{eq:reduced_dyn_particle} has no singularities.

The second-order differential equation~\eqref{eq:reduced_dyn_particle}
describes the reduced dynamics on $\Gamma$. Its state space is the
cylinder ${\cal C}=\{(s,\dot s) \in [\mathbb{R}]_{2\pi} \times \mathbb{R}\}$, which
is diffeomorphic to $\Gamma$ through the diffeomorphism $T : {\cal
	C}\to \Gamma$, $(s,\dot s) \mapsto (\sigma(s), \sigma'(s) \dot
s)$. The results of this paper will show that small variations of the
parameters $a, b$, and of the direction of the vector $B(q)$, have
major effects on the Lagrangian structure of the reduced dynamics, to
the point that the reduced dynamics may not admit a Lagrangian
structure at all. In particular, we distinguish four cases.
\begin{figure}
	\psfrag{1}[c]{\footnotesize $a,b,0$}
	\psfrag{2}[c]{\footnotesize $a,0$}
	\psfrag{3}[c]{\footnotesize $b$}
	\psfrag{4}[c]{\ \footnotesize $0$}
	\psfrag{5}[c]{\footnotesize $b$}
	\psfrag{6}[c]{\footnotesize $a$}
	\psfrag{a}[c]{(a)}
	\psfrag{b}[c]{(b)}
	\psfrag{c}[c]{(c)}
	\psfrag{d}[c]{(d)}
	\psfrag{t}{\small $\theta$}
	\centerline{\includegraphics[width=.99\textwidth]{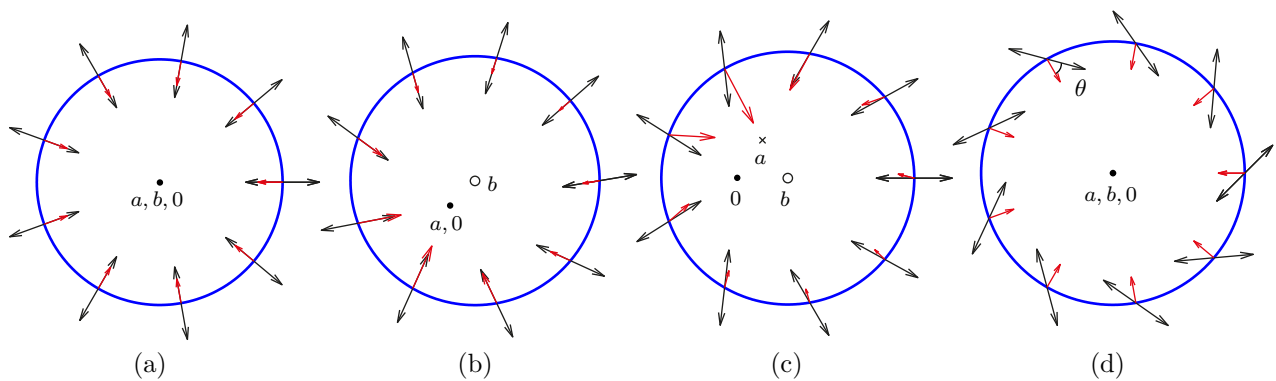}}
	\caption{A material particle immersed in a gravitational field is
		constrained via feedback control to lie on a unit circle. The figure
		depicts four situations corresponding to different values of the
		vectors $a$ and $b$ representing the centre of the gravitational
		field and the centre of the circle.  Black arrows display the
		direction of the control force, while red arrows represent the
		gravitational force. In part (a), the control force is orthogonal to
		the VHC, and the VHC is equivalent to an ideal holonomic
		constraint. The reduced dynamics are Lagrangian and mechanical. In
		part (b), the control force is not orthogonal to the VHC, and the
		VHC is no longer equivalent to an ideal holonomic constraint. Yet,
		the reduced dynamics are still Lagrangian and mechanical. In part
		(c), the reduced dynamics are Lagrangian but not mechanical. In part
		(d), the control force imparts on acceleration on the particle as it
		moves along the circle, and the reduced dynamics are neither
		Lagrangian nor mechanical.}
	\label{fig:point_mass}
\end{figure}

{\bf Case 1:} $a=b=0$. The gravity force and the control force are
parallel to each other, and they are both orthogonal to the circle
$h^{-1}(0)$. See Figure~\ref{fig:point_mass}(a). The gravity force is
compensated by the control force, and it does not affect the reduced
dynamics.  Moreover, the work of the control force $F$ on virtual
displacements $\xi \in T_q h^{-1}(0)$ is identically zero.  Thus, the
VHC $h(q)=0$ is analogous to a holonomic constraint satisfying the
Lagrange-d'Alembert principle of classical mechanics
(see~\cite{ArnoldClassicalMech}). In mechanics, such holonomic
constraint is said to be {\em ideal.} In this setting, we expect the
reduced dynamics to be Lagrangian and, indeed, the reduced
motion~\eqref{eq:reduced_dyn_particle} is $\ddot s=0$, which is a
Lagrangian mechanical system with Lagrangian function $L(s,\dot s) =
(1/2) \dot s^2$. Modulo a constant, this function can be obtained by
restricting the original Lagrangian $\mathcal{L}$ on $\Gamma$, i.e., $L(s,\dot
s) = \mathcal{L}(q,\dot q)\big|_{q=\sigma(s),\dot q=\sigma'(s)\dot s} +
c$. This is precisely what happens in mechanics with ideal holonomic
constraints.

{\bf Case 2:} $a=0$, $b\neq 0$.  The gravity force is parallel
to the control force, but the control force is no longer orthogonal to
the circle $h^{-1}(0)$. See Figure~\ref{fig:point_mass}(b). Now the
work of the control force on virtual displacements $\xi \in T_q
h^{-1}(0)$ is not zero, so one can no longer draw an analogy between
the VHC $h(q)=0$ and an ideal holonomic constraint.  Nonetheless, the
results of this paper will show that the reduced dynamics are a
Lagrangian mechanical system with Lagrangian function $L(s,\dot s) =
(1/2) M(s)\dot s^2$, for a suitable smooth function $M : [\mathbb{R}]_{2\pi}
\to \mathbb{R}$. Since the control force makes work on virtual displacements,
it is no longer true that $L(s,\dot s) = \mathcal{L}(q,\dot
q)\big|_{q=\sigma(s),\dot q=\sigma'(s)\dot s} + c$.

{\bf Case 3:} $a,b \neq 0$. Now the gravity force is no
longer parallel to the control force, and the control force is not
orthogonal to the circle $h^{-1}(0)$.  See
Figure~\ref{fig:point_mass}(c). In this case, the gravity force
affects the reduced dynamics, and the work of the control force on
virtual displacements $\xi \in T_q h^{-1}(0)$ is not zero.  We will
see that for certain values of $a,b$, the reduced dynamics are
Lagrangian, but not mechanical. In other words, the Lagrangian
function of the reduced dynamics cannot be written in the form kinetic
minus potential energy. We will also see that the qualitative
properties of the reduced motion are drastically different than in
cases 1 and 2.

{\bf Case 4:} $a=b=0$, $B(q) = R_\theta\, q$, where $R_\theta$
is a counter-clockwise planar rotation by angle $\theta \in
(-\pi/2,\pi/2), \theta \neq 0$.  See
Figure~\ref{fig:point_mass}(d). In this case, the gravity force is
orthogonal to the circle $h^{-1}(0)$ and it does not affect the
reduced dynamics, while the control force has a constant angle
$\theta$ to the normal vector to the circle.  We shall show that the
reduced dynamics are not Lagrangian.

The example of a material particle on a plane illustrates that the
reduced dynamics induced by VHCs can exhibit very different properties
than the dynamics of a mechanical system subject to a holonomic
constraint.  A number of questions arise in this context:
\begin{enumerate}[Q1]
	\item When are the reduced dynamics Lagrangian and mechanical (i.e.,
	such that the Lagrangian has the form $L=T-V$)?
	
	\item When are the reduced dynamics Lagrangian but not mechanical?
	
	\item Can one expect a Lagrangian structure to exist generically for
	the reduced dynamics, or rather, is it an exceptional property?
	
	\item When a Lagrangian structure exists, what qualitative properties
	can one expect for the reduced dynamics?
\end{enumerate}

This paper will provide answers to these questions. We will return
to the particle example in Section~\ref{sec:examples}.


\section{Preliminaries on Virtual Holonomic Constraints}
\label{sec:Prelim}

In order to generalize the setup of the example in
Section~\ref{sec:introductory_example}, and to introduce the notions
needed to formulate the inverse Lagrangian problem, in this section we
review  basic material taken from~\cite{Maggiore-2012}. Consider a
Lagrangian control system with $n$ DOF and $n-1$ actuators modelled as
\[
\ELeq{\mathcal{L}}{q}=B(q)\tau.
\]
In the above, $q = (q_1, \ldots,q_n) \in \mathcal{Q}$ is the configuration
vector. We assume that each component $q_i$, $i \in \n$, is either a
linear displacement in $\mathbb{R}$, or an angular displacement in
$[\mathbb{R}]_{T_i}$, for some $T_i >0$ (often, $T_i$ is equal to $2 \pi$).
With this assumption, the configuration manifold $\mathcal{Q}$ is a
generalized cylinder, and $T\mathcal{Q}$ is the Cartesian product $T\mathcal{Q} = \mathcal{Q}
\times \mathbb{R}^n$. The term $B(q) \tau$ represents external forces
produced by the control vector $\tau \in \mathbb{R}^{n-1}$. We assume that
$B:\mathcal{Q} \to \mathbb{R}^{n\times (n-1)}$ is smooth and $\rank B(q)=n-1$ for all
$q \in \mathcal{Q}$. Further, the function $\mathcal{L}: T \mathcal{Q} \to \mathbb{R}$ is assumed to
be smooth and to have the special form
$\mathcal{L}(q,\dot{q})=\frac{1}{2}\dot{q}^TD(q)\dot{q}-P(q)$, where $D(q)$,
the generalized mass matrix, is symmetric and positive definite for
all $q\in\mathcal{Q}$. We will assume that there exists a left annihilator of
$B$ on $Q$. That is to say, there exists a smooth function
$B^\perp:\mathcal{Q}\to\mathbb{R}^{1\times n}$ which does not vanish and is such that
$B^\perp(q)B(q)=0$ on $\mathcal{Q}$. With the above mentioned assumptions, the
Lagrangian control system takes on the following standard form
\begin{equation}
\label{eq:ELsystem}
D(q)\ddot{q}+C(q,\dot{q})\dot{q}+\nabla P(q)= B(q)\tau.
\end{equation}
\begin{definition}[\cite{Maggiore-2012}]
	\label{def:VHC}
	A {\bf virtual holonomic constraint (VHC)} of order $n-1$ for system
	\eqref{eq:ELsystem} is a relation $h(q)=0$, where
	$h:\mathcal{Q}\to\mathbb{R}^{n-1}$ is a smooth function which has a regular value at
	$0$, i.e., $\mbox{rank}(dh_q)=n-1$ for all $q\in
	h^{-1}(0)$, and is such that the set
	\begin{equation}
	\label{eq:ConstManifold}
	\Gamma = \{(q,\dot{q}): h(q)=0,\,dh_q\dot{q}=0\}
	\end{equation}
	\noindent is controlled invariant. That is to say, there exists a
	smooth feedback $\tau: \Gamma \to \mathbb{R}^{n-1}$ such that $\Gamma$ is
	positively invariant for the closed-loop system. The set $\Gamma$ is
	called the {\bf constraint manifold} associated with $h(q)=0$. A
	VHC is said to be {\bf stabilizable} if there exists a smooth feedback
	$\tau(q,\dot{q})$ that asymptotically stabilizes $\Gamma$. Such a
	stabilizing feedback is said to {\bf enforce the VHC} $h(q)=0$.
\end{definition}

Since, for each $q \in h^{-1}(0)$, the set of velocities $\{\dot q \in
\mathbb{R}^n : dh_q \dot q=0\}$ is the tangent space $T_q h^{-1}(0)$, it
follows that the constraint manifold $\Gamma$ is the tangent bundle of
$h^{-1}(0)$, $\Gamma = T h^{-1}(0)$. Therefore, the controlled
invariance of $\Gamma$ in Definition~\ref{def:VHC} means that if $q(0)
\in h^{-1}(0)$ and $\dot q(0) \in T_{q(0)} h^{-1}(0)$, then
through the application of a suitable smooth feedback, the
configuration trajectory $q(t)$ can be made to satisfy the VHC
$h(q)=0$ for all $t\geq 0$.

By the preimage theorem~\cite{GuiPol:74}, if $h(q)=0$ is a VHC of
order $n-1$, then the set $h^{-1}(0)$ is a one-dimensional
embedded submanifold of $\mathcal{Q}$. Therefore, $h^{-1}(0)$ is a regular
curve without self-intersections which is diffeomorphic to either the
real line $\mathbb{R}$ or the unit circle $\mathbb{S}^1$.
\begin{definition}[\cite{Maggiore-2012}]
	A relation $h(q)=0$, where $h : \mathcal{Q} \to \mathbb{R}^{n-1}$ is a smooth
	function, is a {\bf regular VHC} of order $n-1$
	for~\eqref{eq:ELsystem} if system~\eqref{eq:ELsystem} with output
	function $e=h(q)$ has well-defined vector relative degree
	$\{2,\dots,2\}$ everywhere on the constraint manifold given
	in~\eqref{eq:ConstManifold}.
\end{definition}

A regular VHC is a VHC. Indeed, the condition that the output function
$e =h(q)$ has vector relative degree $\{2,\ldots,2\}$ implies
(see~\cite{Isi95}) that $\rank(dh_q)=n-1$ for all $q \in
h^{-1}(0)$. Moreover, the zero dynamics manifold exists and it
coincides with $\Gamma$, implying that $\Gamma$ is controlled
invariant.  Regular VHCs enjoy two important properties.  First, under
mild assumptions (see~\cite{Maggiore-2012}), regular VHCs are
stabilizable by input-output feedback linearizing feedback. Indeed, we
have $\ddot e = \mu(q,\dot q) + A(q) u$, where
\[
\mu(q,\dot{q}):= -dh_qD^{-1}(q)[C(q,\dot{q})\dot{q}+\nabla P(q)]+ 
\mathcal{H}h(q,\dot{q}), 
\] 
where $\mathcal{H}h(q,\dot{q})=[\dot{q}^\top
\mbox{Hess}(h_1(q))\dot{q},\dots,\dot{q}^\top
\mbox{Hess}(h_{n-1}(q))\dot{q}]^\top$, and $\mbox{Hess}(h_i(q))$ is
the Hessian matrix of $h_i$ at $q$, and
\[
A(q):= dh_qD^{-1}(q)B(q).
\]
The matrix $A(q)$ is the decoupling matrix associated with the output
function $e =h(q)$. The regularity of the VHC $h(q)=0$ implies that
$A(q)$ is invertible for all $q \in \Gamma$ and therefore, by
continuity, it is also invertible in a neighbourhood of $\Gamma$.  The
input-output feedback linearizing controller
\begin{equation}
\label{eq:feedbackVHC}
\tau(q,\dot{q}) = A^{-1}(q)[-\mu(q,\dot{q})-k_1e-k_2\dot{e}], \ k_1,k_2>0,
\end{equation}
yields $\ddot e + k_2 \dot e + k_1 e =0$, so that $(e,\dot e)=(0,0)$
is an asymptotically stable equilibrium.  Under mild
assumptions~\cite{Maggiore-2012}, this property implies that $\Gamma$
is asymptotically stable.

The second useful property of regular VHCs is that they induce
well-defined reduced dynamics. Specifically, the dynamics on $\Gamma$
(i.e., the zero dynamics associated with the output $e=h(q)$) are
given by a second-order unforced system. In order to find the reduced
dynamics, we follow a procedure presented in~\cite{Dame-2012a}. We
first pick a regular parametrization $\sigma:\Theta \to \mathcal{Q}$ of the
curve $h^{-1}(0)$, where $\Theta = \mathbb{R}$ if $h^{-1}(0) \simeq \mathbb{R}$,
while $\Theta = [\mathbb{R}]_T$, $T>0$, if $h^{-1}(0) \simeq \mathbb{S}^1$. The map
$\sigma:\Theta \to \sigma(\Theta) = h^{-1}(0)$ is a
diffeomorphism. Therefore, the global differential $d\sigma: T \Theta
\to T h^{-1}(0)$, $(s,\dot s) \mapsto (\sigma(s),\sigma'(s)\dot s)$ is
a diffeomorphism as well. Since, as we argued earlier, $Th^{-1}(0) =
\Gamma$, we conclude that $T \Theta \simeq \Gamma$. Next,
multiplying~\eqref{eq:ELsystem} on the left by $B^\perp(q)$ we obtain
\[
B^\perp D \ddot q + B^\perp (C \dot q + \nabla P) =0.
\]
The dynamics on $\Gamma$ are found by restricting the above equation
to $\Gamma$. To this end, we use the fact that $d \sigma : T\Theta \to
\Gamma$ is a diffeomorphism, and we let $q = \sigma(s)$, $\dot
q=\sigma'(s) \dot s$, and $\ddot q = \sigma'(s) \ddot s + \sigma''(s)
\dot s^2$. By so doing, we obtain
\begin{equation}\label{eq:sys}
\ddot s = \Psi_1(s) + \Psi_2(s) \dot s^2,
\end{equation}
where 
\[
\Psi_1(s) = -\left.\frac{B^\perp\nabla P}{B^\perp D\sigma'}\right|_{q=\sigma(s)},
\]
\[
\Psi_2(s) = -\left.\frac{B^\perp
	D\sigma''+\sum_{i=1}^{n}B^\perp_i\sigma'{}^\top Q_i\sigma'}{B^\perp
	D\sigma'}\right|_{q=\sigma(s)},
\]
and where $B^\perp_i$ is the $i$\textsuperscript{th} component of
$B^\perp$ and
$(Q_i)_{jk}=1/2(\partial_{q_k}D_{ij}+\partial_{q_j}D_{ik}-\partial_{q_i}D_{kj})$.

The unforced autonomous system~\eqref{eq:sys} represents the
reduced dynamics of system~\eqref{eq:ELsystem} when the regular VHC of
order $n-1$, $h(q)=0$, is enforced.  The state space
of~\eqref{eq:sys} is $T\Theta = \Theta \times \mathbb{R}$ which, as we
have seen, is diffeomorphic to $\Gamma$. The set $T \Theta$ is a plane
if $h^{-1}(0) \simeq \mathbb{R}$, and a cylinder if $h^{-1}(0)$ is a Jordan
curve.  The reduced dynamics for the material particle example in
Section~\ref{sec:introductory_example} have precisely the
form~\eqref{eq:sys}.

\section{Main Results}
\label{sec:InvLagProb}

In this section we formulate and solve the main problem investigated
in this paper for a two-dimensional system of the form~\eqref{eq:sys},
with state space $\mathcal{X} = T\Theta$, with $\Theta = \mathbb{R}$ or $[\mathbb{R}]_T$,
$T>0$. The functions $\Psi_i : \Theta \to \mathbb{R}$, $i=1,2$, are assumed
to be smooth. We begin by defining precisely the Lagrangian structures
under consideration.
\begin{definition}\label{def:EL} 
	System~\eqref{eq:sys} is said to be:
	
	\begin{enumerate}[(a)]
		
		\item {\bf Euler-Lagrange (EL) with Lagrangian $\mathbf{L}$} if there
		exists a smooth Lagrangian function $L: \mathcal{X} \to \mathbb{R}$ such that the
		following two properties hold:
		\begin{enumerate}[(i)]
			
			\item The Lagrangian $L$ is nondegenerate, i.e., $ \partial^2 L /
			\partial \dot s^2 >0$ for all $(s,\dot s) \in \mathcal{X}$.
			
			\item All solutions $(s(t),\dot s(t))$ of~\eqref{eq:sys} satisfy the
			Euler-Lagrange equation
			\begin{equation}\label{eq:EL}
			\frac{d}{dt} \frac{\partial L}{\partial \dot s}(s(t),\dot s(t)) -
			\frac{\partial L}{\partial s}(s(t),\dot s(t)) =0
			\end{equation}
			for all $t$ in their maximal interval of definition. 
		\end{enumerate}

		\item {\bf Mechanical} if it is EL with Lagrangian $L(s,\dot s) =
		(1/2) M(s) \dot s^2 - V(s)$, where $M : \Theta\to (0,\infty)$, $V
		:\Theta \to \mathbb{R}$ are smooth.
		
		\item {\bf Singular Euler-Lagrange (SEL) with Lagrangian $\mathbf{L}$}
		if there exists a smooth Lagrangian function $L: \mathcal{X} \to \mathbb{R}$ such
		that property~(ii) of part (a) holds. Moreover, if $L$ is any
		function satisfying property~(ii) of part (a) and such that $\partial^2 L / \partial
		\dot s^2$ is not identically zero, then
		\begin{enumerate}[(i)$'$]
			
			\item $L$ is degenerate, i.e., $\partial^2 L / \partial \dot s^2$
			has zeros.
		\end{enumerate}
	\end{enumerate}
\end{definition}
\begin{remark}\label{rem:EL}
	It is well-known that EL systems with Lagrangian $L$ are Hamiltonian
	with Hamiltonian function given by the Legendre transform of $L$ (see,
	e.g.,~\cite{ArnoldClassicalMech}).  On the other hand, while SEL
	systems have a Lagrangian structure, they are generally not
	Hamiltonian because the Legendre transform of $L$ may not be
	well-defined. Moreover, SEL systems are not mechanical since, by
	definition, $\partial^2 L / \partial \dot s^2 = M(s) >0$ for a
	mechanical system.  If $L$ is the Lagrangian of an EL system of the
	form~\eqref{eq:sys}, the Euler-Lagrange equation~\eqref{eq:EL} defines
	a smooth vector field on $\mathcal{X}$ which coincides with~\eqref{eq:sys}.
	Indeed, requirement (i) in Definition~\ref{def:EL}(a) ensures that the
	coefficient of $\ddot s$ in~\eqref{eq:EL} is not zero, and
	therefore~\eqref{eq:EL} defines a smooth vector field on $\mathcal{X}$.
	Moreover, by uniqueness of solutions of~\eqref{eq:sys} and requirement
	(ii) in Definition~\ref{def:EL}(a), the local phase flow of this
	vector field must coincide with the local phase flow
	of~\eqref{eq:sys}. Hence, the vector field arising from~\eqref{eq:EL}
	must coincide with~\eqref{eq:sys}.  On the other hand, we will show in
	the proof of Proposition~\ref{prop:SEL} (see Remark~\ref{rem:SEL})
	that, for a SEL system, the Euler-Lagrange equation~\eqref{eq:EL}
	gives rise to the equation
	\[
	\alpha(s,\dot s) \left[\ddot  s - \Psi_1(s) - \Psi_2(s)\dot s^2\right]=0,
	\]
	where $\alpha$ is a smooth function with zeros. It follows from this
	identity that the Euler-Lagrange equation does not give rise to a
	well-defined vector field on $\mathcal{X}$, and the collection of its
	solutions contains, but is not equal to the
	collection of solutions of~\eqref{eq:sys}. We will illustrate this
	fact with an example in Section~\ref{sec:examples}.  Finally, we
	remark that the requirement, in Definition~\ref{def:EL}(c), that
	$\partial^2 L / \partial \dot s^2$ is not identically zero guarantees
	that the Euler-Lagrange equation~\eqref{eq:EL} gives rise to a
	second-order differential equation.
\end{remark}
%
%
%
%
%
%
%
%
%
%
%
%

{\sc \bf Inverse Lagrangian Problem (ILP).}  {\em Find necessary and
	sufficient conditions under which system~\eqref{eq:sys} is,
	respectively, EL, mechanical, or SEL.}

In order to present the solution of ILP, we let $\tilde \Psi_i: \mathbb{R}
\to \mathbb{R}$, $i=1,2$, be defined as $\tilde \Psi_i(x) :=\Psi_i([x]_T)$,
and we define the {\bf virtual mass} $\tilde M :\mathbb{R} \to (0,\infty)$
and {\bf virtual potential} $\tilde V : \mathbb{R} \to \mathbb{R}$ as
\begin{equation}
\label{eq:M_V}
\begin{aligned}
& \tilde{M}(x) =
\exp\Big(-2\int_0^{x}\tilde{\Psi}_2(\tau)\,d\tau\Big), \\
& \tilde{V}(x)= -\int_0^{x}\tilde{\Psi}_1(\tau)\tilde{M}(\tau)\,d\tau.
\end{aligned}
\end{equation}
We now present the main results of this paper.
\begin{theorem}[{\bf Solution to ILP - Part 1}]\label{thm:ILPsolution:part1}
	If $\Theta = \mathbb{R}$, then system~\eqref{eq:sys} with state space $\mathcal{X} =
	T \Theta$ is mechanical, with $M = \tilde M$ and $V=\tilde V$, where
	$\tilde M,\tilde V$ are defined in~\eqref{eq:M_V}.
\end{theorem}
\begin{proof}
	By straightforward computation, the Euler-Lagrange equation with
	Lagrangian $L(s,\dot s) = (1/2) \tilde M(s) \dot s^2- V(s)$ produces
	equation~\eqref{eq:sys}.  \qquad\end{proof}
\begin{remark}
	In~\cite{ShiPerWit05,ShiRobPerSan06}, the authors presented an
	integral of motion for a system of the form~\eqref{eq:sys} which is
	similar to the total energy $E_0(s,\dot s) = (1/2) \tilde M(s) \dot
	s^2 + \tilde V(s)$, but depends on initial conditions.
\end{remark}
\begin{theorem}[{\bf Solution to ILP - Part 2}]\label{thm:ILPsolution:part2}
	If $\Theta = [\mathbb{R}]_T$, then the following statements about
	system~\eqref{eq:sys} with state space $\mathcal{X} = T \Theta$ are
	equivalent:
	
	\begin{enumerate}[(i)]
		
		\item System~\eqref{eq:sys} is EL.
		
		\item System~\eqref{eq:sys} is mechanical.
		
		\item The functions $\tilde M$ and $\tilde V$ in~\eqref{eq:M_V} are
		$T$-periodic. 
	\end{enumerate}
	
	Moreover, if~\eqref{eq:sys} is EL, then the Lagrangian function $L : T
	[\mathbb{R}]_T \to \mathbb{R}$ is given by $L(s,\dot s) = (1/2) M(s) \dot s^2 -
	V(s)$, where $M:[\mathbb{R}]_T \to (0,\infty)$ and $V: [\mathbb{R}]_T \to \mathbb{R}$ are
	the unique smooth functions such that $\tilde M = M \circ \pi$ and
	$\tilde V = V \circ \pi$.
\end{theorem}
\begin{remark}
	The sufficiency part of the theorem was proved
	in~\cite{Maggiore-2012,Dame-2012a}, but we present it in
	Section~\ref{subsec:virMass} for completeness.
\end{remark}

\begin{theorem}[{\bf Solution to ILP - Part 3}]\label{thm:ILPsolution:part3}
	If $\Theta = [\mathbb{R}]_T$, then the following statements about
	system~\eqref{eq:sys} with state space $\mathcal{X} = T \Theta$ are equivalent:
	\begin{enumerate}[(i)]
		
		\item System~\eqref{eq:sys} is SEL.
		
		\item The function $\tilde M$ is $T$-periodic, while $\tilde V$ is not
		$T$-periodic.
	\end{enumerate}
	
	Moreover, if~\eqref{eq:sys} is SEL, then the Lagrangian function $L :
	T[\mathbb{R}]_T \to \mathbb{R}$ is the unique smooth function such that
	$L(\pi(x),\dot x) = \tilde L(x,\dot x)$ for all $(x,\dot x) \in \mathbb{R}
	\times \mathbb{R}$, where
	\begin{equation}\label{eq:Lagrangian:SEL}
	\begin{aligned}
	&\tilde L(x,\dot x) = - \sin(2 \pi f_0 \tilde E_0(x,\dot x)) +
	\sqrt{2f_0\tilde{M}(x)} \, \pi\dot{x} \, \times \\
	&\Bigg[\cos(2\pi
	f_0\tilde{V}(x))\, \mathbf{C}\Big(\sqrt{2f_0\tilde{M}(x)}\,\dot{x}\Big)-\sin(2\pi
	f_0\tilde{V}(x))\, \mathbf{S}\Big(\sqrt{2f_0\tilde{M}(x)}\, \dot{x}\Big)\Bigg],
	\end{aligned}
	\end{equation}
	where $f_0 = 1 / \tilde V(T)$, $\tilde E_0(x,\dot x) = (1/2) \tilde
	M(x) \dot x^2 + \tilde V(x)$, and $\mathbf{C}(\cdot)$, $\mathbf{S}(\cdot)$ are the
	Fresnel cosine and sine integrals, defined as $\mathbf{C}(x) = \int_0^x
	\cos(\pi t^2/2) dt$, $\mathbf{S}(x) = \int_0^x \sin(\pi t^2/2) dt$.
\end{theorem}
\begin{remark}
	The periodicity conditions in Theorems~\ref{thm:ILPsolution:part2}
	and~\ref{thm:ILPsolution:part3} are coordinate invariant. In
	Proposition~\ref{prop:coord_transformations} we show that they are
	invariant under vector bundle isomorphisms $T[\mathbb{R}]_{T_1} \to T[\mathbb{R}]_{T_2}$,
	$(s,\dot s) \mapsto (\varphi(s),\varphi'(s)\dot s)$, where
	$T_1,T_2>0$.
\end{remark}
\begin{remark}\label{rem:exceptionality}
	Theorems~\ref{thm:ILPsolution:part2} and~\ref{thm:ILPsolution:part3}
	show that, when $\Theta = [\mathbb{R}]_T$ (which, in the setup presented in
	Section~\ref{sec:Prelim}, corresponds to the situation when the VHC
	$h(q)=0$ is a Jordan curve) the property of~\eqref{eq:sys} being
	either EL or SEL is {\em exceptional,} in that it is not satisfied by
	a generic system of the form~\eqref{eq:Lagrangian:SEL} with state
	space $T\Theta$. Indeed, in order for~\eqref{eq:sys} to be EL or SEL
	it is required at a minimum that $\tilde M(x)$ be $T$-periodic, which
	corresponds to requiring that the $T$-periodic function $\tilde
	\Psi_2:\mathbb{R} \to \mathbb{R}$ has zero average. In other words, the set
	$\{\tilde{\Psi}_2:\mathbb{R}\to\mathbb{R}| \int_0^T\tilde{\Psi}_2(\tau)d\tau=0\}$
	has measure zero in the set of all smooth $T$-periodic and real-valued
	functions defined on the real line.
\end{remark}
\begin{remark}\label{rem:lit}
	Having presented the main results of this paper, we now return to the
	literature on the IPLM and place the theorems above in this
	context. First off, it is a matter of straightforward computation to
	check that the reduced dynamics~\eqref{eq:sys} always satisfy the
	Helmholtz conditions and, as such, system~\eqref{eq:sys} is
	automatically guaranteed to be {\em locally} Lagrangian. This fact is
	known since the work of Darboux~\cite{Darboux-1894}. For the existence
	of {\em global} Lagrangian structures, Theorem 5.8
	in~\cite{Takens-1979} indicates that when the state space
	of~\eqref{eq:sys} is $\mathbb{S}^1 \times \mathbb{R}$, from the existence of a local
	Lagrangian structure one cannot deduce the existence of a global such
	structure. As a matter of fact, Theorems~\ref{thm:ILPsolution:part2}
	and~\ref{thm:ILPsolution:part3} show that a global Lagrangian
	structure generally does not exist. The work of Anderson and
	Duchamp~\cite[Theorem~4.2]{Anderson-1980} provides necessary and
	sufficient conditions under which a locally variational source form
	(in our context, the reduced dynamics~\eqref{eq:sys}) is globally
	variational (in our context, globally Lagrangian). The conditions are
	in terms of the vanishing of a cohomology class which is guaranteed to
	exist but for which there is no systematic construction method. The
	criterion in~\cite{Anderson-1980} is therefore indirect.  It might be
	possible to use the methodology of \cite{Anderson-1980} to obtain a
	different proof of some of the results presented above, the
	application of Theorem~4.2 in \cite{Anderson-1980} to the context of
	this paper is far from trivial, and it is unclear whether that
	formalism allows one to distinguish between the existence of EL and
	SEL structures.  In this sense, to the best of our knowledge the
	results stated above are not contained in existing literature. Owing
	to the very specific form of the differential equation we investigate,
	we take a direct route to solving the inverse Lagrangian problem for
	the reduced dynamics arising from a VHC. The results stated above
	present necessary and sufficient conditions which are explicit and
	checkable.
\end{remark}

In the next two sections we prove Theorems~\ref{thm:ILPsolution:part2}
and~\ref{thm:ILPsolution:part3} assuming that $\Theta = [\mathbb{R}]_T$.  We
now provide an outline of the arguments that follow.

{\sc Outline of proofs of Theorems~\ref{thm:ILPsolution:part2}
	and~\ref{thm:ILPsolution:part3}.}

\begin{enumerate}[Step 1] \itemsep1pt \parskip0pt \parsep0pt
	
	\item In Section~\ref{sec:lift}, we define a lifted system, $\ddot x =
	\tilde \Psi_1(x) + \tilde \Psi_2(x) \dot x^2$, with state space
	$\mathbb{R}^2$. In Lemma~\ref{thm:lem1}, we show that trajectories of the
	lifted system are related to trajectories of system~\eqref{eq:sys}
	through the map $d \pi$, where $\pi(x) = [x]_T$.
	
	\item In Lemma~\ref{thm:lem2}, we show that solutions of the
	Euler-Lagrange equation~\eqref{eq:EL} are related through the map $d
	\pi$ to solutions of the Euler-Lagrange equation with Lagrangian
	$\tilde L = L \circ d \pi$.
	
	\item
	Leveraging Lemmas~\ref{thm:lem1} and~\ref{thm:lem2}, in
	Proposition~\ref{prop:equivR_RmodT} we show that~\eqref{eq:sys} is
	EL or SEL if and only if the lifted system is EL or SEL with
	a Lagrangian $\tilde L(x,\dot x)$ which is $T$-periodic with respect to $x$.
	
	\item In Section~\ref{subsec:virMass}, we find necessary and
	sufficient conditions for the existence of a Lagrangian $\tilde L$
	for the lifted system which enjoys the periodicity property of
	Proposition~\ref{prop:equivR_RmodT}. In Proposition~\ref{prop:EL} we
	show that in order for a function $\tilde L(x,\dot x)$ which is
	nondegenerate and $T$-periodic with respect to $x$ to be a
	Lagrangian for the lifted system, it is necessary and sufficient
	that $\tilde M$ and $\tilde V$ in~\eqref{eq:M_V} are
	$T$-periodic. This result proves
	Theorem~\ref{thm:ILPsolution:part2}.
	
	\item In Lemma~\ref{lem:MVprop}, we find expressions for $\tilde
	M(x + n T)$, $\tilde V(x + n T)$, $n \in \mathbb{Z}$.
	
	\item Using Lemma~\ref{lem:MVprop}, in Proposition~\ref{prop:SEL}, we
	prove that the lifted system is SEL with a Lagrangian $\tilde
	L(x,\dot x)$ which is $T$-periodic with respect to $x$ if and only
	if $\tilde M$ in~\eqref{eq:M_V} is $T$-periodic, while $\tilde V$
	isn't. In light of Proposition~\ref{prop:equivR_RmodT}, this proves
	Theorem~\ref{thm:ILPsolution:part3}.
\end{enumerate}

\section{Lift of ILP to $\mathbf{\mathbb{R}^2}$}
\label{sec:lift}

Let $\pi: \mathbb{R} \to [\mathbb{R}]_T$ be defined as $\pi(x) =[x]_T$, and let
$\bar{\pi}: T \mathbb{R} \to T [\mathbb{R}]_T$ denote the global differential of $\pi$,
$\bar{\pi}:=d \pi$, so that $\bar{\pi}(x,\dot x) = ([x]_T,d \pi_x \dot x) =
([x]_T,\dot x)$. Given two functions $f:[\mathbb{R}]_T \to \mathbb{R}$ and $F : T
[\mathbb{R}]_T \to \mathbb{R}$, we define their {\bf lifts} to be functions $\tilde
f := f \circ \pi : \mathbb{R} \to \mathbb{R}$, and $\tilde F :=F \circ \bar{\pi} : T \mathbb{R}
\to \mathbb{R}$, as in the following commutative diagrams:
%
%
%
\begin{center}
	\begin{tikzpicture}[node distance=1.5cm, auto]
	\node (Mtilde) {$\mathbb{R}$};
	\node (M) [node distance=3cm,right of=Mtilde] {$[\mathbb{R}]_T$};
	\node (R) [below of=Mtilde] {$\mathbb{R}$};
	\draw[->] (Mtilde) to node {$\pi$} (M);
	\draw[->,dashed] (Mtilde) to node [swap] {$\tilde f$} (R);
	\draw[->] (M) to node {$f$} (R);
	\end{tikzpicture}
	\begin{tikzpicture}[node distance=1.5cm, auto]
	\node (Xtilde) {$T\mathbb{R}$};
	\node (X) [node distance=3cm,right of=Xtilde] {$T [\mathbb{R}]_T$};
	\node (R) [below of=Xtilde] {$\mathbb{R}$};
	\draw[->] (Xtilde) to node {$\bar{\pi}:=d \pi$} (X);
	\draw[->,dashed] (Xtilde) to node [swap] {$\tilde F$} (R);
	\draw[->] (X) to node {$F$} (R);
	\end{tikzpicture}
\end{center} 
\noindent If $\tilde L : T \mathbb{R} \to \mathbb{R}$ is a smooth function, its
associated Euler-Lagrange equation is
\begin{equation}\label{eq:EL:lifted}
\frac{d}{dt} \frac{\partial \tilde L}{\partial \dot x} -
\frac{\partial \tilde L}{\partial x}=0.
\end{equation}
Finally, we define the lift of system~\eqref{eq:sys} as
\begin{equation}\label{eq:sys:lift}
\ddot x = \tilde \Psi_1(x) + \tilde \Psi_2(x) \dot x^2,
\end{equation}
where $\tilde \Psi_1$ and $\tilde \Psi_2$ are the lifts of $\Psi_1$
and $\Psi_2$.  The state space of the above
differential equation is $\tilde{\mathcal{X}} = T\mathbb{R}$.  We will apply
to system~\eqref{eq:sys:lift} the terminology of
Definition~\ref{def:EL}, whereby $L$ will be replaced by $\tilde L$.
\begin{lemma}
	\label{thm:lem1}
	The vector field of equation~\eqref{eq:sys} is $\bar{\pi}$-related to the
	vector field of~\eqref{eq:sys:lift}. Therefore, pair $(s(t),\dot
	s(t))$ is a solution of~\eqref{eq:sys} if and only if there exists a
	solution $( x(t),\dot x(t))$ of~\eqref{eq:sys:lift} such that
	$(s(t),\dot s(t)) = \bar{\pi}(x(t),\dot x(t))$.
\end{lemma}
\begin{proof}
	The vector fields of system~\eqref{eq:sys} and
	system~\eqref{eq:sys:lift} are given by
	\[
	\begin{aligned}
	& F: \mathcal{X} \rightarrow T\mathcal{X}, \quad (s,\dot{s})\mapsto \dot{s}
	\frac{\partial}{\partial s} + \left( \Psi_1(s)+\Psi_2(s)\dot{s}^2
	\right) \frac{\partial}{\partial \dot s} \\
	& \tilde F: \tilde{\mathcal{X}} \rightarrow T\tilde{\mathcal{X}}, \quad (x,\dot{x})\mapsto
	\dot{x} \frac{\partial}{\partial x} + \left( \tilde \Psi_1(x)+\tilde
	\Psi_2(x)\dot{x}^2 \right) \frac{\partial}{\partial \dot x}.
	\end{aligned}
	\]
	Recall that $\pi(x) = [x]_T$, and $\bar{\pi}(x,\dot x) =
	([x]_T,d\pi_{x}\dot x) =([x]_T,\dot x)$. For all $(x,\dot x) \in
	\tilde{\mathcal{X}}$, the differential $d\bar{\pi}_{(x,\dot x)}: T_{(x,\dot x)}
	\tilde{\mathcal{X}} \to T_{\bar{\pi}(x,\dot x)} \mathcal{X}$ is the identity map
	\[
	d \bar{\pi}_{(x,\dot x)} \left(v_1 \frac{ \partial} {\partial x} + v_2
	\frac{\partial} { \partial \dot x}\right) = v_1 \frac{\partial}
	{\partial s} + v_2 \frac{\partial} {\partial \dot s}.
	\]
	We thus have
	\[
	\begin{aligned}
	d\bar{\pi}_{(x,\dot{x})}\tilde F(x,\dot{x}) &=\dot{x}
	\frac{\partial}{\partial s} + \left(
	\tilde{\Psi}_1(x)+\tilde{\Psi}_2(x)\dot{x}^2\right) \frac{\partial} {
		\partial \dot s}\\
	&= \left( \dot s \frac{\partial}{\partial s} + \left(
	\Psi_1(s)+\Psi_2(s)\dot s^2\right) \frac{\partial}{\partial \dot
		s}\right)\Bigg|_{(s,\dot s)= \bar{\pi}(x,\dot x)} \\
	&=F \circ \bar{\pi}(x, \dot x),
	\end{aligned}
	\]
	proving that $F$ and $\tilde F$ are $\bar{\pi}$-related.
	Since $\bar{\pi}$ is surjective, by~\cite[Proposition~9.6]{Lee13}, a
	pair $(s(t),\dot s(t))$ is a solution of~\eqref{eq:sys} if and only if
	there exists a solution $( x(t),\dot x(t))$ of~\eqref{eq:sys:lift}
	such that $(s(t),\dot s(t)) = \bar{\pi}(x(t),\dot x(t))$.
	\qquad\end{proof}
\begin{lemma}
	\label{thm:lem2}
	Let $I \subset \mathbb{R}$ be an open interval, and $s:I \to [\mathbb{R}]_T$,
	$x:I \to \mathbb{R}$ be $C^1$ signals such that $(s(t),\dot s(t)) =
	\bar{\pi}(x(t),\dot x(t))$ for all $t \in I$. Then, the pair $(s(t),\dot
	s(t))$ satisfies the Euler-Lagrange equation~\eqref{eq:EL} with smooth
	Lagrangian $L: T [\mathbb{R}]_T \to \mathbb{R}$ if and only if the pair $(x(t),\dot
	x(t))$ satisfies the lifted Euler-Lagrange
	equation~\eqref{eq:EL:lifted} with smooth Lagrangian $\tilde L = L
	\circ \bar{\pi}$.
\end{lemma}
\begin{proof}
	We have
	\[
	d\tilde{L}_{(x(t),\dot{x}(t))}=d(L\circ\bar{\pi})_{(x(t),\dot{x}(t))}=dL_{\bar{\pi}(x(t),\dot{x}(t))}\circ
	d\bar{\pi}_{(x(t),\dot{x}(t))}=dL_{\bar{\pi}(x(t),\dot x(t))}.
	\]
	Using the fact that the partial derivatives of $\tilde L$ and $L$ are
	the components of $d\tilde L_{(x,\dot x)}$ and $d L_{(s,\dot s)}$,
	respectively, we have
	\[
	\frac{\partial\tilde{L}}{\partial x}(x(t),\dot{x}(t)) = \frac{\partial
		L}{\partial s}(\bar{\pi}(x(t),\dot{x}(t))), \quad
	\frac{\partial\tilde{L}}{\partial
		\dot{x}}(x(t),\dot{x}(t))=
	\frac{\partial L}{\partial \dot{s}}
	(\bar{\pi}(x(t),\dot{x}(t))),
	\]
	from which it follows that the Euler-Lagrange
	equation~\eqref{eq:EL:lifted} with Lagrangian $\tilde L = L \circ
	\bar{\pi}$ is satisfied along $(x(t),\dot{x}(t))$ if and only if the
	Euler-Lagrange equation~\eqref{eq:EL} with Lagrangian $L$ is
	satisfied along $(s(t),\dot{s}(t))=\bar{\pi}(x(t),\dot{x}(t))$.
	\qquad\end{proof}
%
%
%
%
%
%
\begin{proposition}
	\label{prop:equivR_RmodT}
	The following statements are equivalent
	\begin{enumerate}[(i)]
		
		\item System~\eqref{eq:sys} with state space $\mathcal{X} = T [\mathbb{R}]_T$ is EL
		(resp., SEL) with Lagrangian $L$.
		
		\item System~\eqref{eq:sys:lift} with state space $\tilde{\mathcal{X}} =T \mathbb{R}$
		is EL (resp., SEL) with Lagrangian $\tilde{L}=L\circ \bar{\pi}$.
		
	\end{enumerate}
\end{proposition}
\begin{proof}
	Let $\tilde L = L \circ \bar{\pi}$. Then, by the reasoning used in the
	proof of Lemma~\ref{thm:lem2}, it is easy to see that $(\partial^2
	\tilde L / \partial \dot x^2)(x,\dot x) = (\partial^2 L / \partial
	\dot s^2) (\bar{\pi}(x,\dot x))$.  Therefore, $L$ is nondegenerate
	(respectively, degenerate) if and only if $\tilde{L}$ is nondegenerate
	(respectively, degenerate). Now, suppose that system \eqref{eq:sys} is
	EL (respectively, SEL) with Lagrangian $L$. Consider an arbitrary
	solution of \eqref{eq:sys:lift}, namely, $(x(t),\dot{x}(t))$, where $x
	:I \to \mathbb{R}$ is $C^1$ and $I \subset \mathbb{R}$ is an open interval. By Lemma
	\ref{thm:lem1}, $(s(t),\dot s(t)):=\bar{\pi}(x(t),\dot{x}(t))$ is a
	solution of \eqref{eq:sys}, and thus satisfies the Euler-Lagrange
	equation \eqref{eq:EL}. By Lemma \ref{thm:lem2}, $(x(t),\dot{x}(t))$
	satisfies the Euler-Lagrange equation with Lagrangian
	$\tilde{L}=L\circ \bar{\pi}$. Since $( x(t), \dot x(t) )$ is an arbitrary
	solution of~\eqref{eq:sys:lift}, and since $\bar{\pi}: T \mathbb{R} \to T [\mathbb{R}]_T$
	is onto, system~\eqref{eq:sys:lift} is EL (respectively, SEL) with
	Lagrangian $\tilde{L}=L\circ \bar{\pi}$. The proof that
	if~\eqref{eq:sys:lift} is EL (respectively, SEL) with Lagrangian
	$\tilde{L}=L\circ \bar{\pi}$, then~\eqref{eq:sys} is EL (respectively, SEL)
	with Lagrangian $L$ is analogous. We consider an arbitrary solution
	$(s(t),\dot s(t))$ of~\eqref{eq:sys}, and we let $(x(t),\dot x(t))$ be
	a solution of~\eqref{eq:sys:lift} such that $(s(t),\dot s(t)) =
	\bar{\pi}(x(t),\dot x(t))$. Such a solution exists by Lemma~\ref{thm:lem1}
	and the fact that $\bar{\pi}$ is onto. Thus, $(x(t),\dot x(t))$ is a
	solution of the Euler-Lagrange equation~\eqref{eq:EL:lifted} with
	Lagrangian $\tilde L = L \circ \bar{\pi}$. By
	Lemma~\ref{thm:lem2}, $(s(t),\dot s(t))$ is a solution of the
	Euler-Lagrange equation~\eqref{eq:EL} with Lagrangian $L$. Since
	$(s(t),\dot s(t))$ is an arbitrary solution of~\eqref{eq:sys}, we
	conclude that~\eqref{eq:sys} is EL (respectively, SEL).
	\qquad\end{proof}
%
%
%
\section{Proofs of Main Results}
\label{subsec:virMass} 
By virtue of Proposition~\ref{prop:equivR_RmodT}, solving ILP and
finding a Lagrangian $L$ for system~\eqref{eq:sys} is equivalent to
solving ILP and finding a Lagrangian $\tilde{L}$ for the lifted
system~\eqref{eq:sys:lift} such that $\tilde L = L \circ \bar{\pi}$, for
some smooth $L: T [\mathbb{R}]_T \to \mathbb{R}$. Given a smooth function $\tilde L
: T\mathbb{R} \to \mathbb{R}$, there exists a smooth function $L : T [\mathbb{R}]_T \to
\mathbb{R}$ satisfying $\tilde L = L \circ \bar{\pi}$ if and only if $\tilde L$ is
$T$-periodic with respect to its first argument, i.e., $\tilde
L(x+T,\dot x) = \tilde L(x,\dot x)$ for all $(x,\dot x) \in T\mathbb{R}$. In
this section, we leverage this fact to prove
Theorems~\ref{thm:ILPsolution:part2} and~\ref{thm:ILPsolution:part3}.

\begin{proposition}
	\label{prop:EL}
	The lifted system~\eqref{eq:sys:lift} is EL with a smooth Lagrangian
	$\tilde L: T \mathbb{R} \to \mathbb{R}$ such that $\tilde L(x+T,\dot x) = \tilde
	L(x,\dot x)$ for all $(x,\dot{x})\in T\mathbb{R}$, if and only if the virtual
	mass $\tilde{M}$ and virtual potential $\tilde{V}$
	in~\eqref{eq:sys:lift} are $T$-periodic. If this is the case, then
	system~\eqref{eq:EL} is mechanical with Lagrangian $L = (1/2) M(s)
	\dot s^2 - V(s)$, where $M$ and $V$ are defined through $\tilde M = M
	\circ \pi$, $\tilde V = V \circ \pi$.
\end{proposition}
\begin{proof}
	$(\Leftarrow)$ If $\tilde M$, $\tilde V$ are $T$-periodic, then
	$\tilde L(x,\dot x) = (1/2) \tilde M(x) \dot x^2 - \tilde V(x)$ is
	$T$-periodic with respect to $x$, and
	\[
	\frac{d}{dt} \frac{\partial \tilde L}{\partial \dot x} -
	\frac{\partial \tilde L}{\partial x} = \tilde M(x)
	\big(\ddot{x}-\tilde{\Psi}_1(x)-\tilde{\Psi}_2(x)\dot{x}^2\big).
	\]
	Since $\tilde M>0$, the lifted system is mechanical with Lagrangian $\tilde L$.
	
	$(\Rightarrow)$ Assume that system \eqref{eq:sys:lift} is EL with
	smooth Lagrangian $\tilde L : T \mathbb{R} \to \mathbb{R}$ such that $\tilde
	L(x+T,\dot x) = \tilde L(x,\dot x)$ for all $(x,\dot x) \in T \mathbb{R}$.
	By definition of EL system, $\tilde L$ is nondegenerate, i.e.,
	$\partial^2 \tilde L / \partial \dot x^2 \neq 0$.  Define a smooth
	function $\tilde E : T \mathbb{R} \to \mathbb{R}$ as
	\[
	\tilde E(x,\dot x) := \dot x \frac{\partial \tilde L}{\partial \dot x}
	(x,\dot x)- \tilde L(x,\dot x).
	\]
	By differentiating the expression for $\tilde E$ above along the
	vector field of~\eqref{eq:sys:lift}, it is readily seen that $\tilde
	E$ is an integral of motion for~\eqref{eq:sys:lift}, i.e., $\dot
	{\tilde E}=0$. Consequently, $\tilde E$ must satisfy the first-order
	linear PDE 
	\begin{equation}\label{eq:PDE}
	\frac{\partial \tilde{E}}{\partial x}\dot{x}+\frac{\partial
		\tilde{E}}{\partial \dot{x}}\left(
	\tilde{\Psi}_1(x)+\tilde{\Psi}_2(x)\dot{x}^2 \right)=0.
	\end{equation}
	Its general solution, obtained via the method of
	characteristics~\cite{YehudaPDE}, is
	$\tilde{E}(x,\dot{x})=F(\tilde{E}_0(x,\dot{x}))$, where $F$ is a
	smooth function and
	\[
	\tilde{E}_0(x,\dot{x})=\frac{1}{2}\tilde{M}(x)\dot{x}^2+\tilde{V}(x).
	\]
	Using the definition of $\tilde{E}$, we have
	\[
	\frac{\partial \tilde{E}}{\partial
		\dot{x}}=\dot{x}\frac{\partial^2
		\tilde{{L}}}{\partial \dot{x}^2}
	\]
	for all $(x,\dot{x})\in T \mathbb{R}$. Therefore, 
	\[
	\frac{\partial^2 \tilde{L}}{\partial
		\dot{x}^2}=\tilde{M}(x)F'(\tilde{E}_0(x,\dot{x})).
	\]
	Since $\partial^2 \tilde L / \partial \dot x^2 > 0$ and $\tilde M>0$,
	it follows that $F'(\tilde{E}_0(x,\dot{x}))>0$ for all $(x,\dot{x})\in
	\mathbb{R}^2$, and thus $F$ is strictly increasing. Furthermore, we know that
	$\tilde{E}(x+T,\dot{x})=\tilde{E}(x,\dot{x})$ for all $(x,\dot{x})\in
	\mathbb{R}^2$. Therefore, for all $(x,\dot x)\in T\mathbb{R}$, we have $F(\tilde
	E_0(x+T,\dot x)) = F(\tilde E_0(x,\dot x))$, which implies that
	$\tilde{E}_0(x+T,\dot x) = \tilde{E}_0(x,\dot x)$. Since $\dot x$ is arbitrary, this
	latter identity implies that $\tilde M$ and $\tilde V$ are
	$T$-periodic. Since $\tilde M$ and $\tilde V$ are
	$T$-periodic, then $(1/2)\tilde{M}(\tilde{x})\dot{\tilde{x}}^2-\tilde{V}(\tilde{x})$ is a Lagrangian for the lifted system \eqref{eq:sys:lift}. By Proposition~\ref{prop:equivR_RmodT}, $L(s,\dot{s})=(1/2)M(s)\dot{s}^2-V(s)$ is a Lagrangian for the original system~\eqref{eq:sys}.  \qquad\end{proof}
\begin{lemma}
	\label{lem:MVprop}
	Consider the virtual mass and virtual potential
	in~\eqref{eq:M_V}. For all $n \in \mathbb{Z}$ and all $x \in \mathbb{R}$, the
	following holds:
	\begin{align}
	& \tilde{M}(x +  nT) = \tilde{M}(T)^{n}\tilde{M}(x) \label{eq:M+nT}\\
	& \tilde{V}(x + nT) = \begin{cases} \tilde{M}(T)^{n}\tilde{V}(x)+
	\tilde{V}(T)\displaystyle\frac{\tilde{M}(T)^{n}
		-1}{\tilde{M}(T) - 1}, & \text{if } \tilde{M}(T)\neq 1,\\
	\tilde{V}(x) + n\tilde{V}(T), & \text{if } \tilde{M}(T)=1.
	\end{cases}\label{eq:V+nT}
	\end{align}
\end{lemma}
\begin{proof}
	Using the $T$-periodicity of  $\tilde{\Psi}_1(x)$ and
	$\tilde{\Psi}_2(x)$, it is straightforward to verify that 
	\begin{equation}
	\label{eq:M_x+T}
	\tilde M(x+T)=\tilde{M}(T)\tilde{M}(x). 
	\end{equation}
	By induction, for $k \geq 0$ it holds that $\tilde{M}(x + kT) =
	\tilde{M}(T)^{k}\tilde{M}(x)$. On the other hand, the identity
	$\tilde{M}(x)=\tilde{M}(x-T+T)=\tilde{M}(T)\tilde{M}(x-T)$ results in
	$\tilde{M}(x-T)=\tilde{M}(T)^{-1}\tilde{M}(x)$. By induction, for $k
	\geq 0$ we have $\tilde{M}(x - kT) =
	\tilde{M}(T)^{-k}\tilde{M}(x)$. This proves identity~\eqref{eq:M+nT}
	for all $n \in \mathbb{Z}$.  Turning to $\tilde V$, using the $T$-periodicity
	of $\tilde \Psi_1$ and identity~\eqref{eq:M_x+T}, we have
	\[
	\begin{aligned}
	\tilde{V}(x+T) &=-\int_0^{T}\tilde{\Psi}_1(\tau)\tilde{M}(\tau)\,d\tau-
	\int_T^{T+x}\tilde{\Psi}_1(\tau)\tilde{M}(\tau)\,d\tau \\
	&= \tilde{V}(T) -
	\int_0^{x}\tilde{\Psi}_1(u+T)\tilde{M}(u+T) du \\
	&= \tilde{V}(T)+\tilde{M}(T)\tilde{V}(x).
	\end{aligned}
	\]
	By induction, for $k \geq 0$ we have
	\[
	\tilde{V}(x+kT) = \tilde{M}(T)^{k}\tilde{V}(x)+
	\tilde{V}(T)\{1+\tilde{M}(T)+\cdots+\tilde{M}(T)^{k-1}\}.
	\]
	If $\tilde M(T) \neq 1$, by using the partial sum of the geometric
	series we obtain the first case of identity~\eqref{eq:V+nT}. If
	$\tilde M(T)=1$, then we obtain $\tilde V(x+kT) = \tilde V(x) + k
	\tilde V(T)$, which is the second case of identity~\eqref{eq:V+nT}. To
	prove the identity for negative $n$, we write $\tilde{V}(x-T+T) =
	\tilde{V}(T) + \tilde{M}(T)\tilde{V}(x-T)$, to get $\tilde{V}(x-T) =
	\tilde{M}(T)^{-1}\tilde{V}(x)-\tilde{M}(T)^{-1}\tilde{V}(T)$. By
	induction, for all $ k \geq 0$ we have
	\[
	\tilde{V}(x-kT) = \tilde{M}(T)^{-k}\tilde{V}(x)-\tilde
	M(T)^{-1}\tilde{V}(T)\{1+\tilde{M}(T)^{-1}+\cdots+\tilde{M}(T)^{-(k-1)}\}.
	\]
	If $\tilde M(T)=1$ we obtain the second case of
	identity~\eqref{eq:V+nT}. If $\tilde M(T) \neq 1$, using the partial
	sum of the geometric series and elementary manipulations we arrive at
	the first case of identity~\eqref{eq:V+nT}. In conclusion,
	identity~\eqref{eq:V+nT} holds for all $n \in \mathbb{Z}$.
	\qquad\end{proof}
\begin{proposition}
	\label{prop:SEL}
	The lifted system \eqref{eq:sys:lift} is SEL with a smooth Lagrangian
	$\tilde L : T \mathbb{R} \to \mathbb{R}$ such that $\tilde L(x+T,\dot{x})=\tilde
	L(x,\dot{x})$ for all $(x,\dot{x})\in T \mathbb{R}$, if and only if the
	virtual mass $\tilde{M}(x)$ in~\eqref{eq:M_V} is $T$-periodic, and the
	virtual potential $\tilde{V}(x)$ is not $T$-periodic.
\end{proposition}
\begin{proof}
	$(\Leftarrow)$ Suppose that the virtual mass $\tilde{M}(x)$ is
	$T$-periodic and the virtual potential $\tilde{V}(x)$ is not
	$T$-periodic, so that $\tilde V(T)\neq 0$ and $f_0 = 1/ \tilde V(T)$
	is well-defined. Consider the function $\tilde L : T \mathbb{R} \to \mathbb{R}$
	defined in~\eqref{eq:Lagrangian:SEL}.  With our definition of $f_0$,
	$\tilde L(x,\dot x)$ is $T$-periodic with respect to $x$. Moreover,
	by direct computation, we have
	\begin{equation}\label{EL_equation:SEL}
	\frac{d}{dt} \frac{\partial \tilde L}{\partial \dot x} -
	\frac{\partial \tilde L}{\partial x} = \tilde \alpha(x,\dot x) \left( \ddot x
	- \tilde \Psi_1(x) - \tilde \Psi_2(x) \dot x^2 \right),
	\end{equation}
	where $\tilde \alpha(x,\dot x)= (\partial^2 \tilde L) / (\partial \dot
	x^2) = 2 \pi f_0 \tilde M(x) \cos(2 \pi f_0 \tilde E_0(x,\dot
	x))$. Note first that $\tilde \alpha$ is not identically zero because
	$\tilde V$ is not identically zero (if it were, then $\tilde V$ would
	be $T$-periodic, contradicting our assumption). At the same time, we
	now show that $\tilde \alpha$ has zeros. By assumption, $\tilde
	M(T)=\tilde M(0) =1$ and $\tilde V(T) \neq V(0) =0$. By
	identity~\eqref{eq:V+nT} in Lemma~\ref{lem:MVprop}, $\tilde V(x) \to
	\pm \infty$ as $|x| \to \infty$, and the two limits as $x \to \pm
	\infty$ have opposite signs, which implies that the continuous map
	$\tilde V :\mathbb{R} \to \mathbb{R}$ is onto. Thus, there exists $\bar x \in \mathbb{R}$
	such that $2 \pi f_0 \tilde V(\bar x) = \pi/2$, implying that $\tilde
	\alpha(\bar x,0) =0$. We have shown that $\tilde \alpha$ has zeros,
	which implies that $\tilde L$ is degenerate. By definition, all
	solutions of the lifted system~\eqref{eq:sys:lift} satisfy the
	differential equation $\ddot x = \tilde \Psi_1(x) + \tilde \Psi_2(x)
	\dot x^2$. Therefore, by identity~\eqref{EL_equation:SEL}, any
	solution of~\eqref{eq:sys:lift} satisfies the Euler-Lagrange equation
	with a degenerate Lagrangian $\tilde L$. In order to complete the
	proof that system~\eqref{eq:sys:lift} is SEL, we need to show that if
	$\tilde L'$ is any other Lagrangian for system~\eqref{eq:sys:lift},
	then $\tilde L$ is degenerate, i.e., $\partial^2 \tilde L' / \partial
	\dot x^2$ has zeros. Suppose there exists a nondegenerate Lagrangian
	$\tilde L'$ for system~\eqref{eq:sys:lift}. Then,
	system~\eqref{eq:sys:lift} is EL, which by Proposition~\ref{prop:EL}
	implies that $\tilde V$ is $T$-periodic, a contradiction.

	$(\Rightarrow)$ Suppose that the lifted system \eqref{eq:sys:lift} is
	SEL, and let $\tilde L$ be a degenerate Lagrangian such that $\tilde
	L(x,\dot x)$ is $T$-periodic with respect to $x$, and $\partial^2
	\tilde L / \partial \dot x^2$ has zeros, but it is not identically
	zero. We need to show that $\tilde M(T)=1$, so that $\tilde M$
	in~\eqref{eq:M_V} is $T$-periodic (this fact will imply that $\tilde
	V$ is not $T$-periodic, because if it were so, then by
	Proposition~\ref{prop:EL} the system would be EL). As in the proof of
	Proposition~\ref{prop:EL}, let $\tilde E = \dot x \partial \tilde L /
	\partial \dot x - \tilde L$. Then, $\tilde E$ satisfies the linear
	PDE~\eqref{eq:PDE}, whose general solution is
	$\tilde{E}(x,\dot{x})=F(\tilde{E}_0(x,\dot{x}))$, with
	$\tilde{E}_0(x,\dot{x})=(1/2)\tilde{M}(x)\dot{x}^2+\tilde{V}(x)$. Since
	$\tilde L$ is $T$-periodic with respect to $x$, so is $\tilde
	E$. Therefore, $\tilde{E}(x,\dot{x})=\tilde{E}(x+ nT,\dot{x})$ for all
	$(x,\dot{x})\in T\mathbb{R}$ and all $n \in \mathbb{Z}$. Using
	Lemma~\ref{lem:MVprop}, for all $n \in\mathbb{Z}$ we have
	\[
	F(E_0(x,\dot x)) = F \big( \tilde{E}_0(x
	+nT,\dot{x}) \big)= F \Bigg( \tilde{M}(T)^n\tilde{E}_0(x,\dot{x})+\tilde{V}(T)
	\frac{\tilde{M}(T)^n-1}{\tilde{M}(T)-1} \Bigg).
	\]
	We claim that if $\tilde{M}(T)\neq 1$, then $F$ is a constant
	function. Indeed, for any $p \in \mathrm{Im}(\tilde E_0)$ and any $n
	\in \mathbb{Z}$, we have
	\[
	F(p) = F\Bigg( \tilde{M}(T)^n
	p+\tilde{V}(T)\frac{\tilde{M}(T)^n-1}{\tilde{M}(T)-1} \Bigg).
	\]
	If $\tilde M(T) >1$, taking the limit as $n \to - \infty$ in both
	sides of the identity above we get
	\[
	F(p) = F\Bigg(\frac{-\tilde{V}(T)}{\tilde{M}(T)-1}\Bigg).
	\]
	If $\tilde M(T) <1$, the same identity is obtained by taking the limit
	for $n \to +\infty$. Since the right-hand side of the identity above
	does not depend on $p$, $F: \mathrm{Im}(\tilde E_0) \to \mathbb{R}$ is a
	constant map. Thus, for all $(x,\dot x) \in T\mathbb{R}$ we have 
	\[
	\frac{\partial \tilde E}{\partial \dot x} = \dot x \frac{\partial^2
		\tilde L}{\partial \dot x^2} =0,
	\]
	and so $\partial^2 \tilde L / \partial \dot{x}^2 \equiv 0$,
	contradicting our hypothesis on $\tilde L$.  \qquad\end{proof}
\begin{remark}\label{rem:SEL}
	Since the degenerate Lagrangian $\tilde L(x,\dot x)$
	in~\eqref{eq:Lagrangian:SEL} is smooth and $T$-periodic with respect
	to $x$, there exists a smooth function $L : T [\mathbb{R}]_T \to \mathbb{R}$ such
	that $L\circ \bar{\pi} = \tilde L$. By Lemma~\ref{thm:lem2}, since $\tilde
	\alpha(x,\dot x)$ is $T$-periodic with respect to
	$x$,~\eqref{EL_equation:SEL} implies that $L$ satisfies the identity
	\[
	\frac{d}{dt} \frac{\partial L}{\partial \dot s} - \frac{\partial
		L}{\partial s} = \alpha(s,\dot s) \left( \ddot s - \Psi_1(s) -
	\Psi_2(s) \dot s^2 \right),
	\]
	where $\alpha$ and $\tilde \alpha$ are related through $\tilde \alpha =
	\alpha \circ \bar{\pi}$.
\end{remark}

\section{Characterization of Motion on the Constraint Manifold}
\label{sec:MotionCharac}

In this section we use the results of Section~\ref{sec:InvLagProb} to
investigate the qualitative properties of solutions of the reduced
dynamics~\eqref{eq:sys}  when $h^{-1}(0)$ is a Jordan
curve. In Section~\ref{sec:coord_transformations}, we investigate the
effect of coordinate transformations, and in
Section~\ref{sec:qualitative_properties} we investigate the
qualitative properties of typical trajectories of EL and SEL systems.

\subsection{Effects of coordinate transformations}\label{sec:coord_transformations}

When the set $h^{-1}(0)$ is a Jordan curve, the state space of the
reduced dynamics is a cylinder. The representation of the reduced
dynamics in~\eqref{eq:sys} was derived through a $T$-periodic regular
parametrization of $h^{-1}(0)$. In this section we investigate the
effects of reparametrizations of the curve $h^{-1}(0)$.
Reparametrizing $h^{-1}(0)$ is equivalent to defining a coordinate
transformation $(s,\dot s) \mapsto (\theta,\dot\theta)$ for
system~\eqref{eq:sys}. More precisely, let $T_1, T_2>0$, and let
$\varphi : [\mathbb{R}]_{T_1} \to [\mathbb{R}]_{T_2}$ be a diffeomorphism.  Let
$\pi_i : \mathbb{R} \to [\mathbb{R}]_{T_i}$, $i=1,2$, be defined as $\pi_i(x) =
[x]_{T_i}$.  Consider the smooth dynamical system with state space $T
[\mathbb{R}]_{T_1}$,
\begin{equation}
\label{eq:sys1}
\ddot s = \Psi^1_1(s) + \Psi^1_2(s) \dot s^2,
\end{equation}
and the vector bundle isomorphism $T [\mathbb{R}]_{T_1} \to T [\mathbb{R}]_{T_2}$ defined as
$(s,\dot s) \mapsto (\theta,\dot \theta) = (\varphi(s),\varphi'(s)
\dot s)$. In $(\theta,\dot \theta)$ coordinates, system~\eqref{eq:sys1}
reads
\begin{equation}
\label{eq:sys2}
\ddot \theta = \Psi^2_1(\theta) + \Psi^2_2(\theta) \dot \theta^2,
\end{equation}
where
\[
\begin{aligned}
& \Psi^2_1\circ \varphi  = \varphi' \, \Psi_1^1 \\
& \Psi^2_2\circ \varphi = \frac{\Psi_2^1}{\varphi'} +
\frac{\varphi''}{\varphi'^2}.
\end{aligned}
\]
Associated with the two dynamical systems above we have two lifted systems
\begin{align} 
\label{eq:lifted:sys1}
\ddot x = \tilde \Psi_1^1(x) + \tilde \Psi_2^1(x) \dot x^2 \\
\label{eq:lifted:sys2}
\ddot y = \tilde \Psi_1^2(y) + \tilde \Psi_2^2(y) \dot y^2 
\end{align}
where $\tilde \Psi^i_j := \Psi^i_j \circ \pi_i$, $i,j=1,2$. We also
have virtual mass and virtual potential functions,
\begin{equation}\label{eq:VandM_transformed}
\begin{aligned}
&\tilde M_i(x) = \exp \Big( -2 \int_0^x \tilde \Psi^i_2(\tau) d \tau
\Big),\\
&\tilde V_i(x) = -\int_0^x \tilde \Psi^i_1(\tau) \tilde M_i(\tau) d
\tau,
\end{aligned}
\end{equation}
$i=1,2$. In Proposition~\ref{prop:coord_transformations} we prove that
$\tilde M_1$, $\tilde V_1$ are $T_1$-periodic if and only if $\tilde
M_2$, $\tilde V_2$ are $T_2$-periodic.  This fact is important because
the main results of this paper in Theorems~\ref{thm:ILPsolution:part2}
and~\ref{thm:ILPsolution:part3} are stated in terms of the periodicity
of the functions $\tilde M$ and $\tilde V$ in~\eqref{eq:M_V}. In
Proposition~\ref{prop:conservative}, we show that if, and only if,
$\tilde M_1$ is $T_1$-periodic, there exists $\varphi: [\mathbb{R}]_{T_1} \to
[\mathbb{R}]_{T_2}$ such that $\Psi_2^2=0$, so that~\eqref{eq:sys2} is a one
DOF conservative system.
\begin{proposition}
	\label{prop:coord_transformations}
	There exists a diffeomorphism $\tilde \varphi: \mathbb{R} \to \mathbb{R}$ such that
	the following diagram commutes:
	\begin{equation}\label{eq:comm_diag:tildevarphi}
	\begin{tikzpicture}[node distance=1.5cm, auto,baseline=(current  bounding  box.center)]
	\node (R1) {$\mathbb{R}$};
	\node (R2) [node distance=2cm,right of=R1] {$\mathbb{R}$};
	\node (R1T1) [below of=R1] {$[\mathbb{R}]_{T_1}$};
	\node (R2T2) [below of=R2] {$[\mathbb{R}]_{T_2}$};
	\draw[->] (R1) to node {$\tilde \varphi$} (R2);
	\draw[->] (R1T1) to node [swap] {$\varphi$} (R2T2);
	\draw[->] (R1) to node [swap] {$\pi_1$} (R1T1);
	\draw[->] (R2) to node {$\pi_2$} (R2T2);
	\end{tikzpicture}
	\end{equation}
	Moreover, the lifted
	systems~\eqref{eq:lifted:sys1},~\eqref{eq:lifted:sys2} are related
	through the coordinate transformation $(x,\dot x) \mapsto (y,\dot y) =
	(\tilde \varphi(x),\tilde \varphi'(x) \dot x)$, and the virtual masses
	and virtual potentials in~\eqref{eq:VandM_transformed} are related as
	follows:
	\begin{equation}\label{eq:M1_M2_V1_V2}
	\tilde M_2 \circ \tilde \varphi= \frac{\tilde M_1}{(\tilde
		\varphi')^2} \frac{(\tilde \varphi'(\tilde
		\varphi^{-1}(0)))^2}{\tilde M_1(\tilde \varphi^{-1}(0))}, \quad
	\tilde V_2 = - \frac{(\tilde \varphi'(\tilde
		\varphi^{-1}(0)))^2}{\tilde M_1(\tilde \varphi^{-1}(0))} \left(
	\tilde V_1 - \tilde V_1(\tilde \varphi^{-1}(0)) \right).
	\end{equation}
	Finally, $\tilde M_1$ is $T_1$-periodic if and only if $\tilde M_2$ is
	$T_2$-periodic, and $\tilde V_1$ is $T_1$-periodic if and only if
	$\tilde V_2$ is $T_2$-periodic.
\end{proposition}
\begin{proof}
	The function $\pi_1 : \mathbb{R} \to [\mathbb{R}]_{T_1}$ is a covering
	map~\cite{Lee13}. Since $\varphi: [\mathbb{R}]_{T_1} \to [\mathbb{R}]_{T_2}$ is a
	diffeomorphism, the function $\varphi \circ \pi_1 : \mathbb{R} \to [\mathbb{R}]_{T_2}$ is
	a covering map as well. By the path lifting property of the circle
	(see~\cite[Corollary~8.5]{Lee13}), there exists a map $\tilde \varphi
	: \mathbb{R} \to \mathbb{R}$ such that $\pi_2 \circ \tilde \varphi = \varphi \circ
	\pi_1$, proving that the diagram~\eqref{eq:comm_diag:tildevarphi}
	commutes. We claim that $\tilde \varphi$ is a diffeomorphism. Being
	covering maps, $\pi_1, \pi_2$ are local diffeomorphisms, implying that
	$\tilde \varphi$ is a local diffeomorphism as well. $\tilde \varphi$
	is surjective because $\varphi$ and $\pi_1$ are surjective. Suppose
	$\tilde \varphi(x_1) = \tilde \varphi(x_2)$. Then, $\pi_2 \circ \tilde
	\varphi(x_1) = \pi_2 \circ \tilde \varphi(x_2)$, and therefore $\varphi \circ
	\pi_1(x_1) = \varphi \circ \pi_1(x_2)$.  $\varphi$ is a
	diffeomorphism, so $\pi_1(x_1) = \pi_1(x_2)$, or $x_1 = x_2 + l T_1$,
	for some $l \in \mathbb{Z}$.  Since $\tilde \varphi' \neq 0$ (because $\tilde
	\varphi$ is a local diffeomorphism), it must be that $l=0$, since
	otherwise $\tilde \varphi$ would not be strictly monotonic.  In
	conclusion, $\tilde \varphi$ is bijective, and therefore also a
	diffeomorphism.  The diffeomorphisms $\varphi$ and $\tilde \varphi$
	induce the commutative diagram,
	\begin{equation}\label{eq:comm_diag:tildedvarphi}
	\begin{tikzpicture}[node distance=1.5cm, auto,baseline=(current  bounding  box.center)]
	\node (R1) {$T\mathbb{R}$};
	\node (R2) [node distance=2cm,right of=R1] {$T\mathbb{R}$};
	\node (R1T1) [below of=R1] {$T[\mathbb{R}]_{T_1}$};
	\node (R2T2) [below of=R2] {$T[\mathbb{R}]_{T_2}$};
	\draw[->] (R1) to node {$d \tilde \varphi$} (R2);
	\draw[->] (R1T1) to node [swap] {$d \varphi$} (R2T2);
	\draw[->] (R1) to node [swap] {$d\pi_1$} (R1T1);
	\draw[->] (R2) to node {$d\pi_2$} (R2T2);
	\end{tikzpicture}
	\end{equation}
	in which $d \varphi$ and $d \tilde \varphi$ are vector bundle
	isomorphisms. Let $F_1: [\mathbb{R}]_{T_1} \to T [\mathbb{R}]_{T_1}$ and $F_2: [\mathbb{R}]_{T_2} \to T
	[\mathbb{R}]_{T_2}$ be the vector fields of systems~\eqref{eq:sys1}
	and~\eqref{eq:sys2}, and let $\tilde F_1 : \mathbb{R} \to T \mathbb{R}$, $\tilde F_2
	: \mathbb{R} \to T \mathbb{R}$ be the vector fields of the lifted
	systems~\eqref{eq:lifted:sys1} and~\eqref{eq:lifted:sys2},
	respectively. By Lemma~\ref{thm:lem1}, $d \pi_1 \circ \tilde F_1=F_1
	\circ \pi_1$.  Also, since $d \varphi$ is an isomorphism, $d \varphi
	\circ F_1 = F_2 \circ \varphi$. Using these two identities, we have
	\[
	d\pi_1 \circ \tilde F_1 = F_1 \circ \pi_1 = \big( (d \varphi)^{-1} \circ
	F_2 \circ \varphi \big) \circ \pi_1.
	\]
	Using the diagram~\eqref{eq:comm_diag:tildevarphi} we have $\varphi
	\circ \pi_1 = \pi_2 \circ \tilde \varphi$, so
	\[
	d \varphi \circ d \pi_1 \circ \tilde F_1 = F_2 \circ \pi_2 \circ
	\tilde \varphi.
	\]
	Using the diagram~\eqref{eq:comm_diag:tildedvarphi} and the fact that
	$F_2$ and $\tilde F_2$ are $\pi_2$-related, we have
	\[
	d \pi_2 \circ d \tilde \varphi \circ \tilde F_1 = d \pi_2 \circ \tilde
	F_2 \circ \tilde \varphi.
	\]
	Finally, since $\pi_2$ is a local diffeomorphism, we get $d \tilde
	\varphi \circ \tilde F_1 = \tilde F_2 \circ \tilde \varphi$, proving
	that the vector fields of systems~\eqref{eq:lifted:sys1}
	and~\eqref{eq:lifted:sys2} are $d \tilde \varphi$-related, i.e., the
	coordinate transformation $(y,\dot y) =(\tilde \varphi(x),\tilde
	\varphi'(x) \dot x)$ maps~\eqref{eq:lifted:sys1}
	into~\eqref{eq:lifted:sys2}. We now derive $\tilde M_2$ and $\tilde
	V_2$. Note first that $\tilde \Psi_i^2 \circ \tilde \varphi = \Psi_i^2
	\circ \pi_2 \circ \tilde \varphi = \Psi_i^2 \circ \varphi \circ
	\pi_1$. Also, differentiating the identity $\varphi \circ \pi_1 =
	\pi_2 \circ \tilde \varphi$, and using the fact that $\pi_1 ' =
	\pi_2'=1$, we have $\varphi' \circ \pi_1 = \tilde \varphi'$. Thus,
	\[
	\begin{aligned} 
	\tilde M_2(\tilde \varphi(x)) &= \exp\Bigg( -2 \int_0^{\tilde
		\varphi(x)} \tilde \Psi^2_2(\tau) d \tau \Bigg) = \exp\Big( -2
	\int_{\tilde \varphi^{-1}(0)}^{x} (\Psi_2^2 \circ \varphi \circ
	\pi_1(\tau)) \tilde \varphi'(\tau) d \tau \Big) \\
	&=\exp\Big( -2 \int_{\tilde \varphi^{-1}(0)}^{x} \tilde \Psi_2^1(\tau)
	d \tau \Big)\exp\Big( -2 \int_{\tilde \varphi^{-1}(0)}^{x}
	\frac{\tilde \varphi''(\tau)}{\tilde \varphi'(\tau)} d
	\tau\Big) \\
	& = \frac{\tilde M_1(x)}{(\tilde \varphi'(x))^2} \frac{(\tilde
		\varphi'(\tilde \varphi^{-1}(0)))^2}{\tilde M_1(\tilde \varphi^{-1}(0))}.
	\end{aligned}
	\]
	Similarly, letting $C = (\tilde \varphi'(\tilde \varphi^{-1}(0)))^2 /
	\tilde M_1(\tilde \varphi^{-1}(0))$, for $\tilde V_2$ we have
	\[
	\begin{aligned}
	\tilde V_2(\tilde \varphi(x)) &= - \int_0^{\tilde \varphi(x)} \tilde
	\Psi_1^2(\tau) \tilde M_2(\tau) d \tau \\
	&=-\int_{\tilde \varphi^{-1}(0)}^x \tilde \Psi_1^2( \tilde
	\varphi(\tau)) \tilde M_2( \tilde \varphi(\tau)) \tilde \varphi'(\tau)
	d \tau \\
	& = -C \int_{\tilde \varphi^{-1}(0)}^x \tilde M_1(\tau) \tilde
	\Psi_1^1(\tau) d \tau = -C \tilde V_1(x) + C\tilde V_1(\tilde
	\varphi^{-1}(0)).
	\end{aligned}
	\]
	Finally, since $\varphi: [\mathbb{R}]_{T_1} \to [\mathbb{R}]_{T_2}$ is a diffeomorphism, it
	has degree $\pm 1$. This implies that $\tilde \varphi(x+T_1) = \tilde
	\varphi(x) \pm T_2$.  This fact and the above expressions for
	$\tilde M_2,\tilde V_2$ imply that $\tilde M_2$ (resp., $\tilde V_2$)
	is $T_2$-periodic if and only if $\tilde M_1$ (resp., $\tilde M_2$) is
	$T_1$-periodic.
	\qquad\end{proof}
\begin{proposition}
	\label{prop:conservative}
	Let $T_2>0$ be arbitrary. Then, there exists $\varphi:[\mathbb{R}]_{T_1} \to [\mathbb{R}]_{T_2}$
	such that $\tilde M_2=1$ and $\Psi_2^2=0$ if, and only if, $\tilde
	M_1$ is $T_1$-periodic.
\end{proposition}
\begin{proof}
	$(\Rightarrow)$ Let $T_2>0$ be arbitrary and $\varphi: [\mathbb{R}]_{T_1} \to [\mathbb{R}]_{T_2}$
	be a diffeomorphism. If $\tilde M_2=1$, then $\tilde M_2$ is
	$T_2$-periodic which by Proposition~\ref{prop:coord_transformations}
	implies that $\tilde M_1$ is $T_1$-periodic.
	
	$(\Leftarrow)$ Let $T_2>0$ be arbitrary, and let $\tilde \varphi: \mathbb{R}
	\to \mathbb{R}$ be defined as
	\[
	\tilde \varphi(x) = \lambda \int_0^x \sqrt{\tilde M_1(\tau) d \tau}, \quad
	\lambda:=\frac{T_2}{\int_0^{T_1} \sqrt{\tilde M_1(\tau)} d \tau}.
	\]
	Since $\inf \tilde \varphi'>0$, $\tilde \varphi$ is a diffeomorphism
	$\mathbb{R} \to \mathbb{R}$. Moreover, $\tilde \varphi'$ is $T_1$ periodic, from
	which it is readily seen that $\tilde \varphi(x+T_1) = \tilde
	\varphi(x) + T_2$. For all $ s \in [\mathbb{R}]_{T_1}$, letting $x \in
	\pi_1^{-1}(s)$, we have
	\[ 
	\pi_2 \circ
	\tilde \varphi \circ \pi_1^{-1}(s)=\pi_2 \circ \tilde \varphi(\{x+l
	T_1: l \in \mathbb{Z}\}) = \pi_2 (\{ \tilde \varphi(x)+lT_2: l \in \mathbb{Z}\}) =
	\pi_2(\varphi(x)).
	\]
	Thus, there exists a smooth function $\varphi: [\mathbb{R}]_{T_1} \to [\mathbb{R}]_{T_2}$ such
	that the diagram~\eqref{eq:comm_diag:tildevarphi} commutes. This
	function is a diffeomorphism because $\tilde \varphi$ is such. By
	Proposition~\ref{prop:coord_transformations}, we have
	\[
	\tilde M_2 (\tilde \varphi(x)) = \frac{\tilde M_1(x)}{\lambda^2 \tilde
		M_1(x)} \frac{\lambda^2 \tilde M_1(0)}{\tilde M_1(0)} =1,
	\]
	proving that $\tilde M_2=1$. By~\eqref{eq:VandM_transformed}, it
	follows that $\tilde \Psi_2^2=0$, and also $\Psi_2^2 =0$.
	\qquad\end{proof}

\subsection{Qualitative properties of the reduced dynamics}
\label{sec:qualitative_properties}

Consider again the reduced dynamics
\begin{equation}\label{eq:sys:3}
\ddot s = \Psi_1 (s) + \Psi_2(s) \dot s^2,
\end{equation}
with state space the cylinder $T[\mathbb{R}]_T$. We now characterize the
qualitative properties of ``typical'' solutions of~\eqref{eq:sys:3}.
%
%
\begin{definition}\label{defn:solution_types}
	A solution $(s(t),\dot s(t))$ of~\eqref{eq:sys:3} is said to be:
	\begin{enumerate}[(i)]
		\item A {\bf rotation} of~\eqref{eq:sys:3} if the set $\gamma =
		\text{Im}((s(\cdot),\dot s(\cdot)))$ is homeomorphic to a circle
		$\{(s,\dot s) \in T [\mathbb{R}]_T : \dot s = \text{constant}\}$ via a vector
		bundle isomorphism of the form $(s,\dot s) \mapsto (s,\mu(s) \dot
		s)$, $\mu \neq 0$.
		
		\item An {\bf oscillation} of~\eqref{eq:sys:3} if $\gamma$ is
		homeomorphic to a circle $\{(s,\dot s) \in T [\mathbb{R}]_T: (s,\dot s) =
		\bar{\pi}(x,\dot x), (x,\dot x) \in T\mathbb{R}, x^2+\dot x^2=\text{constant}\}$
		via a vector bundle isomorphism of the form above.
		
		\item A {\bf helix} of~\eqref{eq:sys:3} if $\gamma$ is homeomorphic to
		the set $\{(s,\dot s) \in T [\mathbb{R}]_T: (s,\dot s) = \bar{\pi}(x,\dot x),
		(x,\dot x) \in T\mathbb{R},\dot x^2 + x=\text{constant}\}$ via a vector
		bundle isomorphism of the form above.
		
	\end{enumerate}
\end{definition}
We now discuss the ``typical'' solutions of EL and SEL systems. The
next result for EL systems is taken
from~\cite[Proposition~4.7]{Maggiore-2012}.
\begin{proposition}[\cite{Maggiore-2012}] \label{prop:EL:qualitative} 
	Suppose that the dynamical system \eqref{eq:sys:3} is EL and let $V, M
	: [\mathbb{R}]_T \to \mathbb{R}$ be the unique smooth functions such that $\tilde V = V
	\circ \pi$, $\tilde M = M \circ \pi$, with $\tilde V,\tilde M$ defined
	in~\eqref{eq:M_V}. Let $\underbar{V} = \min_{x \in [0,T]} \tilde
	V(x)$, $\bar V = \max_{x \in [0,T]} \tilde V(x)$.  Then, all solutions
	of \eqref{eq:sys:3} in the set $\{(s,\dot{s})\in T[\mathbb{R}]_T:
	1/2M(s)\dot{s}^2 + V(s)> \overline{V}\}$ are rotations, and almost all
	(in the Lebesgue sense) solutions of \eqref{eq:sys:3} in the set
	$\{(s,\dot{s})\in T [\mathbb{R}]_T: \underline{V}<1/2M(s)\dot{s}^2 +
	V(s)<\overline{V}\}$ are oscillations.
\end{proposition}

Next, a new result concerning SEL systems.
\begin{proposition}\label{prop:SEL:qualitative}
	Suppose that the dynamical system \eqref{eq:sys:3} is SEL. Then,
	almost all solutions of~\eqref{eq:sys:3} are either oscillations or helices.
\end{proposition}
\begin{proof}
	Since~\eqref{eq:sys:3} is a SEL system, by
	Proposition~\ref{prop:conservative} it is diffeomorphic to a one DOF
	conservative system
	\begin{equation}
	\label{eq:sys:unit}
	\ddot{s}=\Psi(s)
	\end{equation}
	with state space $T [\mathbb{R}]_T$, whose associated virtual potential $\tilde
	V(x)=-\int_0^x \tilde \Psi(\tau)d \tau$ (where $\tilde \Psi = \Psi
	\circ \pi$) is {\em not} $T$-periodic, i.e., $\tilde V(T) \neq \tilde
	V(0)=0$.  The lifted system is given by
	\begin{equation}
	\label{eq:sys:unit:lifted}
	\ddot x = \tilde \Psi(x).
	\end{equation}
	In light of Lemma~\ref{thm:lem1}, the solutions of
	systems~\eqref{eq:sys:unit} and~\eqref{eq:sys:unit:lifted} are
	$\bar{\pi}$-related, and to prove the proposition it suffices to show that
	almost all solutions of~\eqref{eq:sys:unit:lifted} are either closed
	curves homeomorphic to $\{(x,\dot x): x^2 + \dot x^2 =
	\text{constant}\}$ or open curves homeomorphic to parabolas $\{(x,\dot
	x) : x + \dot x^2 = \text{constant}\}$. Without loss of generality, we
	assume that $\tilde V(T)>0$. By Lemma~\ref{lem:MVprop}, $\tilde
	{V}(x+nT)= \tilde{V}(x)+n\tilde{V}(T)$ for all $x\in \mathbb{R}$ and all
	$n\in \mathbb{Z}$, implying that $\tilde{V}:\mathbb{R}\to\mathbb{R}$ is onto. Each
	phase curve of \eqref{eq:sys:unit} lies entirely in a level set of
	$\tilde E_0(x,\dot x)=1/2\dot x^2+\tilde V(x)$. 
	%
	%
	By Sard's theorem~\cite{GuiPol:74}, for almost all $h \in \mathbb{R}$,
	$\tilde V' \neq 0$ on the set $\tilde V^{-1}(h)$, which implies that
	the set $\tilde E_0^{-1}(h)$ does not contain equilibria.  Moreover,
	since $\tilde V$ is onto, $\tilde V^{-1}(h)$ is non-empty. Let $h$ be
	such that $\tilde V'\neq 0$ on $\tilde V^{-1}(h)$, and consider the
	set $\Omega_h=\{x\in\mathbb{R}: \tilde V(x) \leq h\}$.  Let
	$\{x_0,\ldots,x_N\}:=\tilde V^{-1}(h)$ be ordered so that $x_i <
	x_{i+1}$.  The sequence is finite since the continuity of $\tilde V$
	and the fact that $\tilde V(x) \to \pm\infty$ as $x \to \pm\infty$
	imply that $x_0 = \inf \tilde V^{-1}(h)$ and $x_N = \sup \tilde
	V^{-1}(h)$ exist and are finite. For all $x < x_0$, $\tilde V(x) < h$, for
	otherwise it would hold that $\inf \tilde V^{-1}(h) < x_0$. Moreover,
	since $\tilde V' \neq 0$ on the set $\tilde V^{-1}(h)$, it follows
	that $\Omega_h$ is the union of disjoint intervals with nonzero
	measure. This latter fact implies that $\Omega_h = (-\infty,x_0]
	\bigcup [x_1,x_2] \bigcup \cdots \bigcup [x_{N-1},x_N]$.  Now we
	apply the classical theory of one DOF conservative
	systems~\cite{ArnoldClassicalMech}, from which we conclude that the
	energy level set $\tilde E_0^{-1}(h)$ is the union of $N+1$
	trajectories. On each band $[x_{2i-1}, x_{2i}] \times \mathbb{R}$,
	$i=1,\ldots, N/2$, the set $\tilde E_0^{-1}(h) \cap \big( [x_{2i-1},
	x_{2i}] \times \mathbb{R}\big)$ is a closed curve homeomorphic to a
	circle $x^2 + \dot x^2 =\text{constant}$ (see also the proof of
	Lemma~3.12 in~\cite{Consolini-2010}). On the band $(-\infty,x_0]
	\times \mathbb{R}$, the set $\tilde E_0^{-1}(h) \cap ((-\infty,x_0]
	\times \mathbb{R})$ is homeomorphic to a parabola $\{(x,\dot x): x+\dot
	x^2=x_0\}$ via the homeomorphism $(x,\dot x) \mapsto \big( x,
	\dot x \sqrt{(-x+x_0)/{2( h - \tilde V(x))}}
	\big)$. \qquad\end{proof}

\begin{remark}
	By virtue of Propositions~\ref{prop:EL:qualitative}
	and~\ref{prop:SEL:qualitative}, EL and SEL systems cannot possess
	limit cycles or asymptotically stable equilibria. Typical solutions of
	an EL system are rocking motions (oscillations) or complete
	revolutions of $s$ (rotations). Typical solutions of a SEL system are
	complete revolutions of $s$ with either a periodic speed profile
	(oscillations) or monotonically increasing or decreasing speed profiles (helices).
\end{remark}

We conclude this section with a slight extension of a result
in~\cite[Proposition 4.1]{Consolini-2012} which shows that certain
systems of the form~\eqref{eq:sys:3} which have no Lagrangian structure
(i.e., they are neither EL nor SEL) possess exponentially stable limit
cycles.
\begin{proposition}[\cite{Consolini-2012}]\label{prop:limit_cycle}
	Consider the dynamical system~\eqref{eq:sys:3}, and assume that either
	$\Psi_1>0$ and $\int_0^T \tilde \Psi_2(\tau) d \tau<0$ or $\Psi_1<0$
	and $\int_0^T \tilde \Psi_2(\tau) d \tau>0$. Define the $T$-periodic
	smooth function $\tilde \nu:\mathbb{R} \to \mathbb{R}$ as
	\[
	\tilde{\nu}(x)=\mbox{sgn}({\Psi}_1)\sqrt{\frac{-2\tilde{M}^{-1}(x)[\tilde{V}(x+T)-\tilde{V}(x)]}{\tilde{M}(T)-1}},
	\]
	and let $\nu: [\mathbb{R}]_T \to \mathbb{R}$ be the unique smooth function such that
	$\tilde \nu = \nu \circ \pi$.  Then the closed orbit ${\cal R} =
	\{(s,\dot{s})\in T[\mathbb{R}]_T\times: \dot{s}=\nu(s)\}$ is exponentially stable
	for~\eqref{eq:sys:3}, with domain of attraction containing the set
	$\mathcal{D}=\{(s,\dot{s})\in T[\mathbb{R}]_T: \mbox{sgn}(\Psi_1)\dot{s}\geq
	0\}$.
\end{proposition}

We omit the proof of this result, since it is almost identical to the
proof of Proposition~4.1 in~\cite{Consolini-2012}. The element of
novelty here is the explicit determination of the limit cycle $\dot s
= \nu(s)$ which is made possible by Lemma~\ref{lem:MVprop}. This latter
result can also be used to show that $\tilde \nu(x)$ is a $T$-periodic
function.

\begin{remark}
	Proposition~\ref{prop:limit_cycle} shows that, generally, the flow of
	the reduced dynamics induced by a VHC does not preserve volume. This
	is in contrast with the flow of Hamiltonian systems which, according
	to the Liouville-Arnold theorem~\cite{ArnoldClassicalMech}, preserves
	volume. Moreover, the sufficient conditions of the proposition are
	expressed in terms of strict inequalities involving continuous
	functions and, as such, they persist under small perturbations of the
	vector field in~\eqref{eq:sys:3}. In other words, the existence of
	stable limit cycles is not an ``exceptional''
	phenomenon. In~\cite{Consolini-2012}, it was shown that the reduced
	dynamics of a bicycle traveling along a closed curve and subject to a
	regular VHC meet the conditions of Proposition~\ref{prop:limit_cycle}.
\end{remark}

\section{Examples}\label{sec:examples}

We now present a number of examples illustrating the results of this
paper. Later, we return to the material particle example of
Section~\ref{sec:introductory_example} and analyze its Lagrangian
structure using Theorems~\ref{thm:ILPsolution:part2}
and~\ref{thm:ILPsolution:part3}.
\begin{example}
	Consider the system
	\[
	\ddot{s} = \frac{1}{2+\cos(s)}[\sin(2s)-\sin(s)\dot{s}^2],
	\]
	where $s\in[\mathbb{R}]_{2\pi}$. The virtual mass and potential are given by
	$\tilde M(x)=9/(\cos x+2)^2$ and $\tilde V(x)=4-18(\cos x+1)/(\cos
	x+2)^2$. Since $\tilde M$ and $\tilde V$ are $2\pi$-periodic, by
	Theorem~\ref{thm:ILPsolution:part2} the system is EL and
	mechanical. By Proposition~\ref{prop:EL:qualitative}, almost all solutions
	are either oscillations or rotations. Figure~\ref{fig:EL} shows the
	phase portrait of the system and two phase curves of the system on the
	phase cylinder $[\mathbb{R}]_{2\pi}\times\mathbb{R}$ corresponding to an oscillation
	and a rotation. 
	\begin{figure}
		\label{fig:EL}
		\includegraphics[scale=0.4]{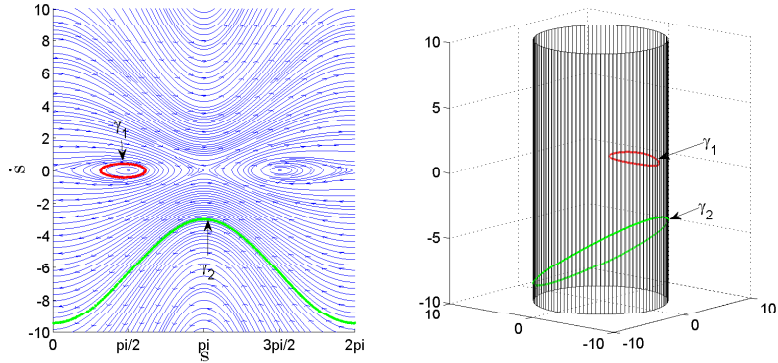}
		\caption{Left: Phase portrait of an EL system. Right: An oscillation
			($\gamma_1$) and a rotation ($\gamma_2$) on the cylinder.}
	\end{figure}
\end{example}

\begin{example}
	For the system
	\[
	\ddot{s} = \cos(s)+0.5+\cos(s)\dot{s}^2,
	\]
	where $s\in[\mathbb{R}]_{2\pi}$, we have 
	\[
	\tilde M(x) = \exp \Big(-2 \int_0^x \tilde \Psi_2(\tau) d \tau \Big) =
	\exp\big(-2 \int_0^x \cos \tau d \tau \big) = \exp ( -2 \sin x),
	\]
	is $2 \pi$-periodic. On the other hand, one can check that $\tilde V(2
	\pi) = -\int_0^{2\pi} (\cos \tau+0.5) \exp(-2 \sin \tau) d \tau \simeq
	7.1615 \neq 0$, so that $\tilde V$ is not $2 \pi$-periodic. By
	Theorem~\ref{thm:ILPsolution:part3}, the system is SEL. By
	Proposition~\ref{prop:SEL:qualitative}, almost all its solutions are
	either oscillations or helices.  Figure~\ref{fig:SEL} shows the phase
	portrait and two typical phase curves on the cylinder, an oscillation
	and a helix.
	\begin{figure}
		\label{fig:SEL}
		\includegraphics[scale=0.4]{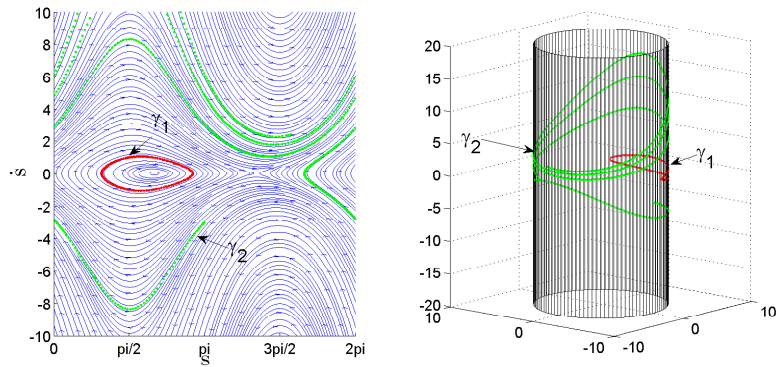} 
		\caption{Left: Phase portrait of a SEL system. Right: An oscillation
			($\gamma_1$) and a helix ($\gamma_2$) on the cylinder.}
	\end{figure}
\end{example}

\begin{example}\label{ex:SEL_contained}
	For the system $\ddot{s}=\lambda$, with $\lambda \neq 0$ and $s \in
	[\mathbb{R}]_T$, we have $\tilde{M}(x)=1$ and $\tilde{V}(x)=-\lambda x$. Since
	$\tilde M$ is $T$ periodic and $\tilde V$ isn't, the system is
	SEL. By Theorem~\ref{thm:ILPsolution:part3}, the Lagrangian is given
	by~\eqref{eq:Lagrangian:SEL}. The Euler-Lagrange equation with this
	Lagrangian reads
	\[
	\ELeq{\tilde{L}}{x} = \frac{2\pi}{\lambda T} \cos\Big(\frac{2\pi}{\lambda T}(\dot{x}^2/2-\lambda x)\Big)(\ddot{x}-\lambda)=0.
	\]
	We see that all solutions of the system $\ddot s = \lambda$ satisfy
	the Euler-Lagrange equation, but there are signals $(x(t),\dot x(t)) =
	(T/4 + kT,0)$, $ k \in \mathbb{Z}$ satisfying the Euler-Lagrange equation
	which do not satisfy the equation $\ddot s=\lambda$. Thus, the
	collection of solutions of a SEL system is contained, but is not equal
	to, the collection of solutions of the associated Euler-Lagrange
	equation.
\end{example}

\begin{example} Consider the system
	\[
	\dot{s} = -\cos(s) - 2 + (\sin(s) + 2)\dot{s}^2
	\]
	with $s \in [\mathbb{R}]_{2 \pi}$. We have $\Psi_1(s) = -\cos( s) -2 <0$ and
	$\int_0^{2\pi} \tilde \Psi_2(\tau) d \tau = \int_0^{2\pi} (\sin \tau
	+2) d \tau = 4 \pi>0$. This latter identity implies that $\tilde M(2
	\pi) \neq 0$, so that $\tilde M$ is not $2 \pi$-periodic, and the
	system is neither EL nor SEL. Moreover, by
	Proposition~\ref{prop:limit_cycle} the system has an exponentially
	stable limit cycle with domain of attraction including
	$\mathcal{D}=\{(s,\dot{s})\in T [\mathbb{R}]_{2 \pi}: \dot{s}\leq
	0\}$. Figure~\ref{fig:non-EL} depicts the phase portrait of the system
	along with the stable limit cycle.
	\begin{figure}
		\label{fig:non-EL}
		\includegraphics[scale=0.4]{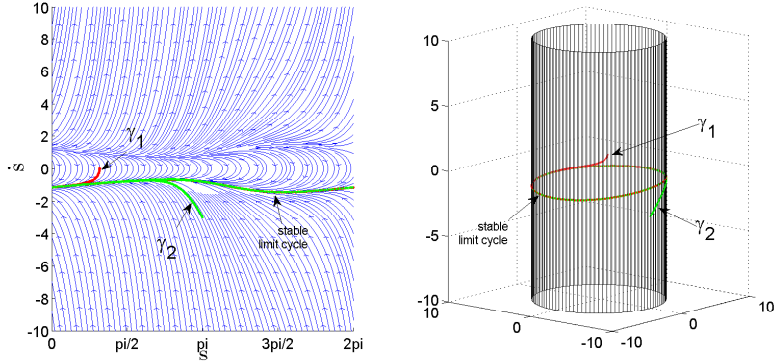}
		\caption{Left: Phase portrait of a non EL system with an attractive
			limit cycle. Right: Two phase curves and the stable limit cycle of
			the system on the phase cylinder.}
	\end{figure}
\end{example}

\begin{example}
	We return to the particle mass example of
	Section~\ref{sec:introductory_example}, in which $s \in [\mathbb{R}]_{2 \pi}$
	and
	\[
	\begin{aligned}
	& \Psi_1(s) = - \frac{\left(a_1 b_2 + a_2 b_1 -a_1 \sin(s)+a_2
		\cos(s)\right)\left(b_1 \cos(s)+b_2 \sin(s) +1\right)}{\left[(b_1
		- a_1 + \cos(s))^2 + (b_2 - a_2 + \sin(s))^2\right]^{3/2}} \\
	& \Psi_2(s) = -\frac{b_1 \sin(s) + b_2 \cos(s)}{b_1 \cos(s) + b_2
		\sin(s) +1},
	\end{aligned}
	\]
	where $a_i,b_i$ are the components of $a,b \in \mathbb{R}^2$.  We now revisit
	the four cases discussed in Section~\ref{sec:introductory_example}.
	
	{\bf Case 1: $a =b=0$.} In this case the reduced dynamics reads as
	$\ddot s=0$, an EL system.
	
	{\bf Case 2: $a=0,b \neq 0$.} Here we have $\Psi_1=0$, implying that
	$\tilde V$ is $2 \pi$-periodic. Moreover, one can check that $\tilde
	M(x) = (4+\cos x)^2/25$, a $2\pi$-periodic function. Thus the reduced
	dynamics are EL.  In this case, the Lagrangian function $L(s,\dot s) =
	1/2 M(s) \dot s^2$ is not equal to the restriction of the Lagrangian
	of the particle mass, $\mathcal{L}(q,\dot q) = (1/2) \|\dot q\|^2 - P(q)$ to
	the constraint manifold.
	
	{\bf Case 3: $a=[1/4 \ \ 3/4]^\top, b=[3/4 \ \ 0]^\top$.} In this
	case $\Psi_2(s)$ is the same as in case 2, but now $\Psi_1(s)\neq
	0$. While $\tilde M$ is $2 \pi$-periodic, one can check that $\tilde
	V(2 \pi) = 0.2762 \neq 0$. The virtual potential is not
	$2\pi$-periodic and thus the system is
	SEL. Figure~\ref{fig:particle:SEL} shows two typical solutions on the
	cylinder, an oscillation and a helix.
	\begin{figure}
		\psfrag{s}{$s$}
		\psfrag{t}{$\dot s$}
		\centerline{\includegraphics[width=.95\textwidth]{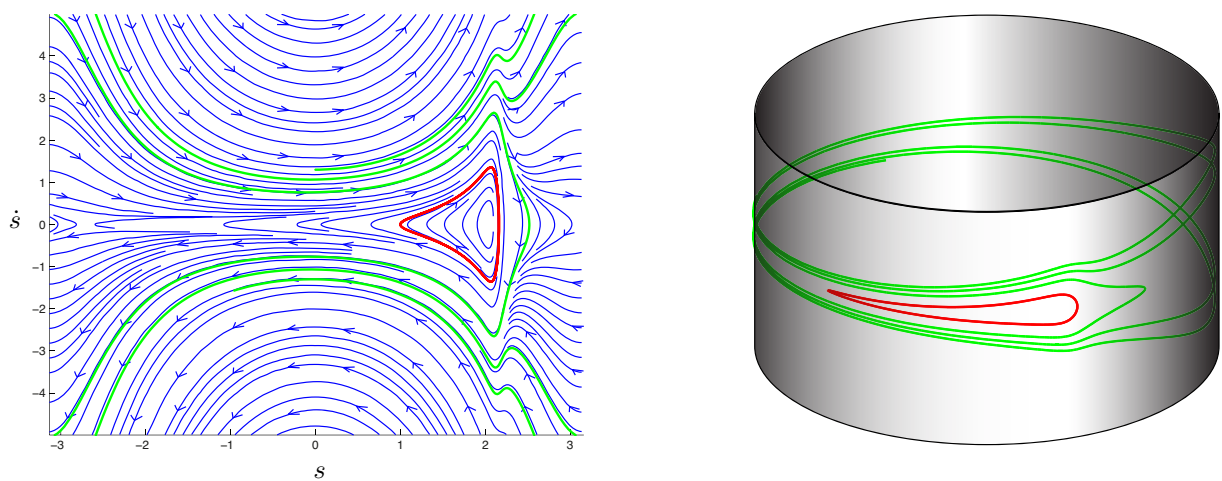}}
		\caption{Left: phase portrait of the particle mass example in case 3. Right: 
			an oscillation and a helix on the phase cylinder.}
		\label{fig:particle:SEL}
	\end{figure}

	{\bf Case 4: $a=b=0, B(q) = R_\theta q$, $\theta \in (-\pi/2,\pi/2)$,
		$\theta \neq 0$.} In this case, the reduced dynamics read as
	\[
	\ddot s = \frac{\tan \theta } 5 - (\tan \theta) \ddot s^2.
	\]
	We have $\tilde M(x) = \exp(-2 \int_0^x -\tan (\theta) d \tau) =
	\exp(2 (\tan \theta) x)$. This is not $2 \pi$-periodic and thus the
	reduced dynamics is neither EL nor SEL.  In sum, arbitrarily small
	variations of the parameters $a,b,\theta$ have drastic effects on the
	Lagrangian properties of the reduced dynamics of the particle.
\end{example}



\bibliography{PhDBiblio}
\end{document}